\documentclass[a4paper]{article}
\usepackage{amsmath}
\usepackage{amsthm}
\usepackage{amsfonts}
\usepackage{amssymb}
\usepackage{mathrsfs}
\usepackage{CJKutf8}
\usepackage{mathtools}
\usepackage{stmaryrd}
\usepackage{amsfonts}
\usepackage{amssymb}
\usepackage{aliascnt}
\usepackage[nottoc]{tocbibind}
\usepackage{hyperref}
\hypersetup{
	colorlinks=true,
	linkcolor=blue
}

\title{Smallness of Faltings heights of CM abelian varieties}
\author{Xunjing Wei}
\begin{document}
	\maketitle

\newcommand{\CColfour}{C_{\mathrm{Col}, 4}}
\newcommand{\CColthree}{C_{\mathrm{Col}, 3}}
\newcommand{\CColtwo}{C_{\mathrm{Col}, 2}}
\newcommand{\CColone}{C_{\mathrm{Col}, 1}}
\newcommand{\oo}{\mathfrak{o}}
\newcommand{\CdiscstCol}{C_\mathrm{disc-st-Col}}
\newcommand{\CdiscstColqua}{C_\mathrm{disc-st-Col-qua}}
\newcommand{\NEp}{(\widetilde{E^*_\Phi})_+}
\newcommand{\CStarkzero}{C_{\mathrm{Stark zero}}}
\newcommand{\Colmu}{\mu_{(E, \Phi)}}
\newcommand{\ColZ}{Z_{(E, \Phi)}}
\newcommand{\hFaltE}{h_{(E, \Phi)}^{\mathrm{Falt}}}
\newcommand{\hCol}{h_{(E, \Phi)}^{\mathrm{Col}}}
\newcommand{\one}{\mathbf{1}}
\newcommand{\tr}{\mathrm{tr}}
\newcommand{\Colm}{m_{(E, \Phi)}}
\newcommand{\ColA}{{A_{(E, \Phi)}^0}}
\newcommand{\NE}{\widetilde{E^*_\Phi}}
\newcommand{\si}{\sigma}
\newcommand{\Stab}{\mathrm{Stab}}
\newcommand{\del}{\delta}
\newcommand{\SumfnotSiegeloroneminusSiegel}{{\sum_{\substack{\rho\colon f(\rho)=0 \\ 0<\re(\rho)<1\\ \rho\ne \be_0\\ \rho\ne 1-\be_0}}
}}
\newcommand{\rhorhobar}{\biggl(\frac{1}{\rho}+\frac{1}{\rhobar}\biggr)}
\newcommand{\SumzetaKnotSiegel}{{\sum_{\substack{\rho\colon \zeta_K(\rho)=0 \\ 0<\re(\rho)<1\\ \rho\ne \be_0}}
}}
\newcommand{\SumfnotSiegel}{{\sum_{\substack{\rho\colon f(\rho)=0 \\ 0<\re(\rho)<1\\ \rho\ne \be_0}}
}}
\newcommand{\sminusoneplusrho}{\frac{1}{s-1+\rho}}
\newcommand{\SumLEF}{{\sum_{\substack{\rho\colon L(\rho, \chi_{E/F})=0 \\ 0<\re(\rho)<1}}
}}
\newcommand{\SumzetaE}{{\sum_{\substack{\rho\colon \zeta_E(\rho)=0 \\ 0<\re(\rho)<1}}
}}
\newcommand{\SumzetaK}{{\sum_{\substack{\rho\colon \zeta_K(\rho)=0 \\ 0<\re(\rho)<1}}
}}
\newcommand{\Sumf}{{\sum_{\substack{\rho\colon f(\rho)=0 \\ 0<\re(\rho)<1}}
}}
\newcommand{\discErho}{\frac{1}{1+\frac{1}{\log|\disc(E)|}-\rho}}
\newcommand{\discErhodiscErhobar}{\biggl(\frac{1}{1+\frac{1}{\log|\disc(E)|}-\rho}+\frac{1}{1+\frac{1}{\log|\disc(E)|}-\rhobar}\biggr)}
\newcommand{\discKrho}{\frac{1}{1+\frac{1}{\log|\disc(K)|}-\rho}}
\newcommand{\discKrhodiscKrhobar}{\biggl(\frac{1}{1+\frac{1}{\log|\disc(K)|}-\rho}+\frac{1}{1+\frac{1}{\log|\disc(K)|}-\rhobar}\biggr)}
\newcommand{\onerho}{\frac{1}{1-\rho}}
\newcommand{\onerhoonerhobar}{\biggl(\frac{1}{1-\rho}+\frac{1}{1-\rhobar}\biggr)}
\newcommand{\srho}{\frac{1}{s-\rho}}
\newcommand{\srhosrhobar}{\biggl(\frac{1}{s-\rho}+\frac{1}{s-\rhobar}\biggr)}
\newcommand{\SumzetaKcomplex}{{\sum_{\substack{\rho\colon \zeta_K(\rho)=0 \\ 0<\re(\rho)<1\\ \im(\rho)>0}}}}
\newcommand{\SumzetaKreal}{{\sum_{\substack{\rho\colon \zeta_K(\rho)=0 \\ 0<\re(\rho)<1\\ \rho\in \R}}
}}
\newcommand{\reg}{\mathrm{reg}}
\newcommand{\Irr}{\mathrm{Irr}}
\newcommand{\g}{\mathfrak{g}}
\newcommand{\chibar}{\overline{\chi}}
\newcommand{\La}{\Lambda}
\newcommand{\vbar}{{\overline{v}}}
\newcommand{\GL}{\mathrm{GL}}
\newcommand{\Ind}{\mathrm{Ind}}
\newcommand{\SumzetaEcomplex}{{\sum_{\substack{\rho\colon \zeta_E(\rho)=0 \\ 0<\re(\rho)<1\\ \im(\rho)>0}}}}
\newcommand{\SumzetaEreal}{{\sum_{\substack{\rho\colon \zeta_E(\rho)=0 \\ 0<\re(\rho)<1\\ \rho\in \R}}
}}
\newcommand{\SumLEFcomplex}{{\sum_{\substack{\rho\colon L(\rho, \chi_{E/F})=0 \\ 0<\re(\rho)<1\\ \im(\rho)>0}}}}
\newcommand{\SumLEFreal}{{\sum_{\substack{\rho\colon L(\rho, \chi_{E/F})=0 \\ 0<\re(\rho)<1\\ \rho\in \R}}
}}
\newcommand{\Sumfcomplex}{{\sum_{\substack{\rho\colon f(\rho)=0 \\ 0<\re(\rho)<1\\ \im(\rho)>0}}}}
\newcommand{\Sumfreal}{{\sum_{\substack{\rho\colon f(\rho)=0 \\ 0<\re(\rho)<1\\ \rho\in \R}}
}}
\newcommand{\SumfrealnotSiegel}{{\sum_{\substack{\rho\colon f(\rho)=0 \\ 0<\re(\rho)<1\\ \rho\in \R \\ \rho\ne \be_0}}
}}
\newcommand{\CHa}{C_{\mathrm{Ha}}}
\newcommand{\oofloor}{{\lfloor \oo \rfloor}}
\newcommand{\CGRHtwo}{C_{\mathrm{GRH},2}}
\newcommand{\CGRHone}{C_{\mathrm{GRH},1}}
\newcommand{\QQ}{\mathfrak{Q}}
\newcommand{\Lie}{\mathrm{Lie}}
\newcommand{\PP}{\mathfrak{P}}
\newcommand{\unr}{\mathrm{unr}}
\newcommand{\OKpcom}{{\OO_{K_\p}}}
\newcommand{\OKploc}{{(\OO_K)_\p}}
\def\bsp{\begin{split}}
\def\esp{\end{split}}
\def\beqs{\begin{equation*}}
\def\eeqs{\end{equation*}}
\def\beq{\begin{equation}}
\def\eeq{\end{equation}}
\newcommand{\metp}{{\mathrm{met},\p}}
\newcommand{\metv}{{\mathrm{met},v}}
\newcommand{\CgoodconsGRH}{C_{\mathrm{good-cons-GRH}}}
\newcommand{\Cgooddiscreflex}{C_{\mathrm{good}-\disc(E^*_\Phi)}}
\newcommand{\CgoodlogdiscE}{C_{\mathrm{good}-\log \disc(E)}}
\newcommand{\Cgoodcons}{C_{\mathrm{good-cons}}}
\newcommand{\CgooddiscE}{C_{\mathrm{good}-\disc(E)}}
\newcommand{\Cdiscstavrqua}{C_{\mathrm{disc-st-avr-qua}}}
\newcommand{\Cdiscstavr}{C_{\mathrm{disc-st-avr}}}
\newcommand{\Cstqua}{C_{\mathrm{st-qua}}}
\newcommand{\Czeroqua}{C_{\mathrm{zero-qua}}}
\newcommand{\Cst}{C_{\mathrm{st}}}
\newcommand{\Cconsst}{C_{\mathrm{cons-st}}}
\newcommand{\Cdiscst}{C_{\mathrm{disc-st}}}
\newcommand{\Ccondst}{C_{\mathrm{cond-st}}}
\newcommand{\Ccons}{C_{\mathrm{cons}}}
\newcommand{\Cdisc}{C_{\mathrm{disc}}}
\newcommand{\Ccond}{C_{\mathrm{cond}}}
\newcommand{\Cabc}{C_\mathrm{abc}}
\newcommand{\Clower}{C_{\mathrm{lower}}}
\newcommand{\Corder}{C_{\mathrm{order}}}
\newcommand{\Cprime}{C_{\mathrm{prime}}}
\newcommand{\Czero}{C_{\mathrm{zero}}}
\newcommand{\Cder}{C_{\mathrm{der}}}
\newcommand{\Csz}{C_{\mathrm{Sz}}}
\newcommand{\rhobar}{\overline{\rho}}
\newcommand{\re}{\mathrm{Re}}
\newcommand{\be}{\beta}
\newcommand{\Ga}{\Gamma}
\newcommand{\Kbar}{\overline{K}}
\newcommand{\Nm}{\mathrm{Nm}}
\newcommand{\ab}{\mathrm{ab}}
\newcommand{\corresponds}{\leftrightarrow}
\newcommand{\into}{\hookrightarrow}
\newcommand{\ga}{\gamma}
\newcommand{\ai}{\alpha}
\newcommand{\normalsupgroup}{\triangleright}
\newcommand{\sendsto}{\rightsquigarrow}
\newcommand{\surto}{\twoheadrightarrow}
\newcommand{\Serre}{\mathfrak{S}}
\newcommand{\Ade}{\mathbb{A}}
\newcommand{\N}{\mathrm{N}}
\newcommand{\Qbar}{\overline{\Q}}
\newcommand{\q}{\mathfrak{q}}
\newcommand{\D}{\mathcal{D}}
\newcommand{\val}{\mathrm{val}}
\newcommand{\e}{\mathbf{e}}
\newcommand{\disc}{\mathrm{disc}}
\newcommand{\G}{\mathbb{G}}
\newcommand{\Res}{\mathrm{Res}}
\newcommand{\ep}{\mathbf{\epsilon}}
\newcommand{\sw}{\mathfrak{sw}}
\newcommand{\Norm}{\mathrm{Norm}}
\newcommand{\f}{\mathfrak{f}}
\newcommand{\im}{\operatorname{Im}}
\newcommand{\FF}{\mathcal{F}}
\newcommand{\hfalt}{h_{\mathrm{Falt}}^{\mathrm{unst}}}
\newcommand{\hst}{h_{\mathrm{Falt}}^{\mathrm{st}}}
\newcommand{\degar}{\deg_{\mathrm{Ar}}}
\newcommand{\Y}{\mathcal{Y}}
\newcommand{\LL}{\mathcal{L}}
\newcommand{\OO}{\mathcal{O}}
\newcommand{\p}{\mathfrak{p}}	
\newcommand{\directlimit}{\varinjlim}
\newcommand{\Spec}{\mathrm{Spec}}
\newcommand{\A}{\mathcal{A}}
\newcommand{\inverselimit}{\varprojlim}
\newcommand{\Hom}{\mathrm{Hom}}
\newcommand{\ten}{\otimes}
\newcommand{\Gal}{\mathrm{Gal}}
\newcommand{\End}{\mathrm{End}}
\newcommand{\C}{\mathbb{C}}
\newcommand{\R}{\mathbb{R}}
\newcommand{\Q}{\mathbb{Q}}
\newcommand{\Z}{\mathbb{Z}}
\newcommand{\F}{\mathbb{F}}
\newcommand{\eup}{é}

\newtheorem{dummy}{***}[section] 
\newcommand{\mynewtheorem}[2]{
	\newaliascnt{#1}{dummy}
	\newtheorem{#1}[#1]{#2}
	\aliascntresetthe{#1}
	\expandafter\def\csname #1autorefname\endcsname{#2}
}

\theoremstyle{plain}
\mynewtheorem{thm}{Theorem}
\mynewtheorem{lem}{Lemma}
\mynewtheorem{prop}{Proposition}
\mynewtheorem{cor}{Corollary}

\theoremstyle{definition}
\mynewtheorem{defn}{Definition}
\mynewtheorem{conj}{Conjecture}
\mynewtheorem{exmp}{Example}
\mynewtheorem{nota}{Notation}
\mynewtheorem{rem}{Remark}
\mynewtheorem{note}{Note}

\makeatletter
\newcommand{\vast}{\bBigg@{4}}
\newcommand{\Vast}{\bBigg@{5}}
\makeatother

\begin{abstract}
We prove that assuming the Colmez conjecture and the ``no Siegel zeros" conjecture, the stable Faltings height of a CM abelian variety over a number field is less than or equal to the logarithm of the root discriminant of the field of definition of the abelian variety times an effective constant depending only on the dimension of the abelian variety. In view of the fact that the Colmez conjecture for abelian CM fields, the averaged Colmez conjecture, and the ``no Siegel zeros" conjecture for CM fields with no complex quadratic subfields are already proved, we prove unconditional analogues of the result above. In addition, we also prove that the logarithm of the root discriminant of the field of everywhere good reduction of CM abelian varieties can be ``small". 
\end{abstract}

\tableofcontents

\section{Introduction}\label{section:Introduction}

Let $E$ be a CM-field, and let $\Phi$ be a CM-type of $E$. Let $A$ be an abelian variety over a number field $K$ such that we have an embedding $i\colon \OO_E\into \End_K(A)$ such that $(A, i)$ has CM-type $\Phi$. It is proved by Colmez in \cite{Colmez} that the stable Faltings height $\hst(A)$ of the abelian variety $A$ depends only on the CM-field $E$ and the CM-type $\Phi$ and not on the abelian variety $A$. We denote it as $\hFaltE$. In \cite{Colmez} Colmez proposed a conjecture relating $\hFaltE$ to the logarithmic derivatives at $s=0$ of certain Artin $L$-functions defined by $(E, \Phi)$. We will refer to this conjecture as the \textit{Colmez conjecture}. The precise statement is as follows:

We define a function $A_{(E, \Phi)}^0$ from $\Gal(\Qbar/\Q)$ to $\C$ by
\[
A_{(E, \Phi)}^0(\sigma)=\frac{1}{[\Gal(\Qbar/\Q):\Stab(\Phi)]}\sum_{\nu \in \Gal(\Qbar/\Q)/\Stab(\Phi)} |\nu \Phi\cap \sigma \nu \Phi|, \forall \sigma\in \Gal(\Qbar/\Q),
\]
where $\Stab(\Phi)$ is the subgroup of $\Gal(\Qbar/\Q)$ consisting of the stabilizers of $\Phi$. 

This function is locally constant and constant on conjugacy classes. Therefore, there is a unique decomposition of $A_{(E, \Phi)}^0$ into $\C$-linear combinations of irreducible Artin characters $\chi$ (i.e. characters $\chi$ of irreducible continuous representations of $\Gal(\Qbar/\Q)$ on finite-dimensional $\C$-vector spaces)
\[
\ColA=\sum_{\chi} m_{(E, \Phi)}(\chi) \chi, \ \ \ \ \ \ \ m_{(E, \Phi)}(\chi) \in \C.
\] 

It can be shown that for any irreducible Artin character $\chi$ such that $m_{(E, \Phi)}(\chi)\ne 0$, the Artin $L$-function $L(s, \chi, \Q)$ is defined and nonzero at $s=0$. We define
\[
\ColZ \coloneqq -\frac{1}{2}g\log(2\pi)+\sum_{\chi} \Colm(\chi) \frac{L'(0, \chi, \Q)}{L(0, \chi, \Q)}
\]
and
\[
\Colmu \coloneqq \sum_{\chi} \Colm(\chi) \log(\f(\chi, \Q)),
\]
where $\f(\chi, \Q)$ is the Artin conductor of the Artin character $\chi$ (a positive integer). 

The Colmez conjecture says that we have
\[
\hFaltE=-\ColZ-\frac{1}{2} \Colmu.
\]

When $E$ is a complex quadratic field, the Colmez conjecture is the same as the classical Chowla-Selberg formula (see for example Page 91 and 92 of \cite{WeilElliptic}), and so it is in fact a theorem. 

Colmez \cite{Colmez} and Obus \cite{Obus} proved that the Colmez conjecture is true if the extension $E/\Q$ is Galois with abelian Galois group (\autoref{thm:Colmez conjecture is true when the CM-field is abelian}). Yuan--Zhang \cite{YuanZhang} and Andreatta--Goren--Howard--Madapusi-Pera \cite{AGHM} independently proved that the Colmez conjecture is true when one averages over all CM-types of a given CM-field (\autoref{thm:averaged Colmez conjecture}). 

Let $-d\in \Z_{\le -2}$ be a fundamental discriminant, so $\disc(\Q(\sqrt{-d}))=d$. Let $\chi_{d}$ be the quadratic character associated to the quadratic field extension $\Q(\sqrt{-d})/\Q$, so $\chi_d(p)=(-d|p)$ for any prime $p$. Let $L(s, \chi_d)$ be the Dirichlet $L$-function of the character $\chi_d$, so $L(s, \chi_d)=\frac{\zeta_{\Q(\sqrt{-d})}(s)}{\zeta_{\Q}(s)}$.

It is known that there is at most one zero of $L(s, \chi_d)$ in the region 
\[
1-\frac{1}{4 \log(d)} \le \re(s)<1, |\im(s)|\le \frac{1}{4\log(d)},
\]
and if such a zero exists it is real and simple.

For any $0<c\le\frac{1}{4}$, we define the \textit{$c$-Siegel zero} of $L(s, \chi_d)$ to be the zero of $L(s, \chi_d)$ in the region 
\[1-\frac{c}{ \log(d)} \le \re(s)<1, |\im(s)|\le \frac{1}{4\log(d)}
\] 
(if it exists). We define the \textit{Siegel zero} of $L(s, \chi_d)$ to be the $\frac{1}{4}$-Siegel zero of $L(s, \chi_d)$. 

The conjecture \nameref{conj:no Siegel zeros for dirichlet L-functions} is as follows:

\begin{conj}[No $\frac{1}{O(1)}$-Siegel zero of $L(s, \chi_d)$]\label{conj:no Siegel zeros for dirichlet L-functions}
There exists some effectively computable absolute constant $\Czero\in \R_{\ge 4}$ such that for any fundamental discriminant $-d\in \Z_{\le -2}$, the Dirichlet $L$-function $L(s, \chi_d)$ has no zeros in the region $1-\frac{1}{\Czero \log(d)} \le \re(s)<1$, $|\im(s)|\le \frac{1}{4\log(d)}$.
\end{conj}

Now let $E$ be an arbitrary CM-field. Let $F$ be the maximal totally real subfield of $E$, and so $E$ is a totally complex quadratic field extension of $F$. Let $\chi_{E/F}$ be the quadratic character associated to the quadratic field extension $E/F$, so for any prime ideal $\p$ of $\OO_F$,
\[
\chi_{E/F}(\p)=\begin{cases} 
1 & \mbox{if } \p \mbox{ splits completely in } \OO_E, \\
-1 & \mbox{if } \p \mbox{ is unramified but does not split completely in } \OO_E, \\
0 & \mbox{if } \p \mbox{ is ramified in } \OO_E.
\end{cases}
\]

Let $L(s, \chi_{E/F})$ be the $L$-function of the character $\chi_{E/F}$, and so $L(s, \chi_{E/F})=\frac{\zeta_{E}(s)}{\zeta_{F}(s)}$. 

Similarly to the case where $E$ is a complex quadratic field, by Lemma 3 of \cite{Stark}, for any CM-field $E$ with maximal totally real subfield $F$, $L(s, \chi_{E/F})$ has at most one zero in the region 
\[
1-\frac{1}{4\log|\disc(E)|}\le\re(s)<1, |\im(s)|\le \frac{1}{4\log|\disc(E)|}.
\] 
If such a zero exists, it is real and simple.

For any $0<c\le\frac{1}{4}$, we define the \textit{generalized $c$-Siegel zero} of $L(s, \chi_{E/F})$ to be the zero of $L(s, \chi_{E/F})$ in the region 
\[1-\frac{c}{ \log|\disc(E)|} \le \re(s)<1, |\im(s)|\le \frac{1}{4\log|\disc(E)|}
\] 
(if it exists). We define the \textit{generalized Siegel zero} of $L(s, \chi_{E/F})$ to be the generalized $\frac{1}{4}$-Siegel zero of $L(s, \chi_{E/F})$. 

The conjecture \nameref{conj:all g no Siegel zeros for L-functions of quadratic characters associated to CM extensions} is as follows:

\begin{conj}[No generalized $\frac{1}{O_g(1)}$-Siegel zero of $L(s, \chi_{E/F})$] \label{conj:all g no Siegel zeros for L-functions of quadratic characters associated to CM extensions}

For any $g\in \Z_{\ge 1}$, there exists some effectively computable constant $\Czero(g)\in \R_{\ge 4}$ depending only on $g$ such that for any CM-field $E$ with maximal totally real subfield $F$ such that $[F:\Q]=g$, the function $L(s, \chi_{E/F})$ has no zeros in the region 
\[
1-\frac{1}{\Czero(g)\log|\disc(E)|} \le \re(s)<1, |\im(s)|\le \frac{1}{4 \log|\disc(E)|}.
\]
\end{conj}

It is proved by Stark (Lemma 9 of \cite{Stark}) that \autoref{conj:no Siegel zeros for dirichlet L-functions} implies \autoref{conj:all g no Siegel zeros for L-functions of quadratic characters associated to CM extensions}. He also proved that \autoref{conj:all g no Siegel zeros for L-functions of quadratic characters associated to CM extensions} is true whenever the CM-field $E$ contains no complex quadratic subfields. 

We show that assuming the Colmez conjecture, the nonexistence of the generalized Siegel zero of $L$-functions of quadratic characters associated to CM extensions is closely related to the stable Faltings height of CM abelian varieties being bounded by the logarithm of the root discriminant of the field of definition. More precisely, we prove the following theorem:

\begin{thm}\label{thm:1}
Suppose that the Colmez conjecture holds. Suppose further that \nameref{conj:no Siegel zeros for dirichlet L-functions} holds. Then for any $g\in \Z_{\ge 1}$, there exist effectively computable constants $C_1(g)>0$, $C_2(g)\in \R$ depending only on $g$ such that 
\begin{equation*}
\begin{split}
\hst(A)\le C_1(g) \frac{1}{[K:\Q]}\log|\disc(K)|+C_2(g),
\end{split}	
\end{equation*}
for any dimension-$g$ abelian variety $A$ defined over a number field $K$ with complex multiplication by $\OO_E$ for some CM-field $E$.
\end{thm}

Since the Colmez conjecture for abelian CM-fields is already proved, and since \autoref{conj:all g no Siegel zeros for L-functions of quadratic characters associated to CM extensions} is true when the CM-field $E$ contains no complex quadratic subfields, we can also prove an unconditional version of the theorem above:

\begin{thm}\label{thm:3}
For any $g\in \Z_{\ge 1}$, there exists effectively computable constants $C_3(g)>0$, $C_4(g)\in \R$ depending only on $g$ such that 
\begin{equation*}
\begin{split}
\hst(A)\le C_3(g) \frac{1}{[K:\Q]}\log|\disc(K)|+C_4(g),
\end{split}	
\end{equation*}
for any dimension-$g$ abelian variety $A$ over a number field $K$ with complex multiplication by $\OO_E$ for some CM-field $E$ such that the extension $E/\Q$ is Galois with abelian Galois group and $E$ does not contain any complex quadratic subfields.
\end{thm}

\begin{rem}
To show that the condition ``$E$ does not contain any complex quadratic subfields" in the hypotheses in \autoref{thm:3} is possible, we give examples of CM fields $E$ containing no complex quadratic subfields such that the extension $E/\Q$ is Galois with abelian Galois group.

Let $n$ be an integer greater than or equal to $3$ such that the group $(\Z/n\Z)^\times$ is a cyclic group and such that $\#(\Z/n\Z)^\times$ divides $4$. (Equivalently, $n=p^k$ or $n=2p^k$ for some odd prime $p$ such that $p\equiv 1 \mod 4$.)

Let $E$ be the $n$-th cyclotomic field $\Q(\mu_{n})$, where $\mu_n$ denotes a primitive $n$-th root of unity. Then $E$ is a CM-field with maximal totally real subfield $F=\Q(\mu_{n}+\mu_{n}^{-1})$. The extension $E/\Q$ is Galois and $\Gal(E/\Q)$ is isomorphic to $(\Z/n\Z)^\times$. Since $\Gal(E/\Q)$ is cyclic and of even order, there is a unique subgroup $H$ of $\Gal(E/\Q)$ of index $2$, and so there is a unique quadratic subfield $K$ of $E$. Let $\iota$ be the nontrivial element of $\Gal(E/F)\subset \Gal(E/\Q)$. Then $\iota$ is the unique element in $\Gal(E/\Q)$ of order $2$. Since $\#\Gal(E/\Q)=\#(\Z/n\Z)^\times$ divides $4$, we have $\iota\in H$. Thus, $K$ is fixed by the element $\iota$. Therefore, $K$ is a real quadratic field and so $E$ contains no complex quadratic subfields. 

More generally, let $E$ be any totally imaginary number field such that the extension $E/\Q$ is Galois with abelian Galois group. Then $E$ is a CM-field (any abelian extension over $\Q$ is either totally real or a CM field). We know that
\[
\Gal(E/\Q)\cong\Z/q_1\Z \times \Z/q_2\Z \times  \cdots \times \Z/q_k\Z,
\]
where $q_1, q_2, \cdots q_m$ are powers of prime numbers ($q_1, q_2, \cdots q_m$ are not necessarily distinct). Suppose further that the number ``$2$" does not appear in $q_1, q_2, \cdots q_m$, i.e. each $q_i$ is either $2^{k_i}$ for $k_i\ge 2$ or a power of an odd prime. Then each component $\Z/q_i\Z$ such that $q_i$ is $2^{k_i}$ ($k_i\ge 2$) contains a unique subgroup $H_i$ of index $2$ and a unique element $\sigma_i$ of order $2$, and $\sigma_i\in H_i$. Let $\iota$ be the nontrivial element of $\Gal(E/F)\subset \Gal(E/\Q)$, where $F$ is the maximal real subfield of $E$. Then $\iota\in H$ for any subgroup $H$ of $\Gal(E/\Q)$ of index $2$. Therefore, $E$ contains no complex quadratic subfields. 
\end{rem}

Since the averaged Colmez conjecture is already proved, we can also prove averaged analogues of the theorems above. 

\begin{thm}\label{thm:2}
For any $g\in \Z_{\ge 1}$, there exists effectively computable constants $C_5(g)>0$, $C_6(g)\in \R$ depending only on $g$ such that 
\begin{equation*}
\begin{split}
&\frac{1}{2}\biggl(\hst(A_1)+\hst(A_2)\biggr)\\
\le & C_5(g) \cdot \frac{1}{2}\biggl(\frac{1}{[K_1:\Q]}\log|\disc(K_1)|+\frac{1}{[K_1:\Q]}\log|\disc(K_2)|\biggr)+C_6(g),
\end{split}	
\end{equation*}
for any pair $A_1, A_2$ of dimension-$g$ abelian varieties defined over number fields $K_1, K_2$ respectively, such that the following holds:
\begin{itemize}
\item There exists a CM-field $E$ of degree $[E:\Q]=2g$ and embeddings $i_1\colon \OO_E\into \End_{K_1}(A_1)$, $i_2\colon \OO_E\into \End_{K_2}(A_2)$ such that $E$ does not contain any complex quadratic subfields and the CM-type $\Phi_1$ of $(A_1, i_1)$ and the CM-type $\Phi_2$ of $(A_2, i_2)$ satisfy:
\[
|\Phi_1\cap\Phi_2|=g-1.
\]
\end{itemize}
\end{thm}

\begin{thm}\label{thm:4}
Let $g$ be a positive integer. Suppose that there exists some effectively computable constant $\Czero(g)\in \R_{\ge 4}$ depending only on $g$ such that for any CM-field $E$ with maximal totally real subfield $F$ such that $[F:\Q]=g$, the function $L(s, \chi_{E/F})$ has no zeros in the region 
\[
1-\frac{1}{\Czero(g)\log|\disc(E)|} \le \re(s)<1, |\im(s)|\le \frac{1}{4 \log|\disc(E)|},
\] 
i.e. the conjecture \nameref{conj:all g no Siegel zeros for L-functions of quadratic characters associated to CM extensions} holds for $g$.

Then there exist effectively computable constants $C_7(g)>0$, $C_8(g)\in \R$ depending only on $g$ such that 
\begin{equation*}
\begin{split}
&\frac{1}{2}\biggl(\hst(A_1)+\hst(A_2)\biggr)\\
\le & C_7(g) \cdot \frac{1}{2}\biggl(\frac{1}{[K_1:\Q]}\log|\disc(K_1)|+\frac{1}{[K_1:\Q]}\log|\disc(K_2)|\biggr)+C_8(g),
\end{split}	
\end{equation*}
for any pair $A_1, A_2$ of dimension-$g$ abelian varieties defined over number fields $K_1, K_2$ respectively, such that the following holds:
\begin{itemize}
\item There exists a CM-field $E$ of degree $[E:\Q]=2g$ and embeddings $i_1\colon \OO_E\into \End_{K_1}(A_1)$, $i_2\colon \OO_E\into \End_{K_2}(A_2)$ such that the CM-type $\Phi_1$ of $(A_1, i_1)$ and the CM-type $\Phi_2$ of $(A_2, i_2)$ satisfy:
\[
|\Phi_1\cap\Phi_2|=g-1.
\]
\end{itemize}
\end{thm}

It might be interesting to know that if we only make use of the (proved) \textit{averaged} Colmez conjecture, then we cannot obtain results stronger than \autoref{thm:2} and \autoref{thm:4} (i.e. the ``average" condition in these theorems cannot be dropped), even if we further assume that the abelian variety over the number field has everywhere good reduction. In particular, we prove the following theorems, which show that the logarithm of the root discriminant of the field of everywhere good reduction of CM abelian varieties can be ``small". 

\begin{thm}\label{prop:assume GRH, field of good reduction can be small}
Assume the Generalized Riemann Hypothesis. For any $g\in \Z_{\ge 1}$, there exist effectively computable constants $C_{13}(g)>0$, $C_{14}(g)\in \R$, such that for any CM-field $E$ with $[E:\Q]=2g$, for any CM-type $\Phi$ of $E$, there exists a number field $K'$ and a CM abelian variety $(A, i\colon \OO_E\into \End_{K'}(A))$ over $K'$ of CM-type $\Phi$ such that the abelian variety $A$ over $K'$ has everywhere good reduction and 
\begin{equation}\label{equ:bound on disc(K') assuming GRH}
\begin{split}
&\frac{1}{[K':\Q]}\log|\disc(K')| \\
\le&  C_{13}(g)\log\log|\disc(E)|+\frac{1}{[E^*_\Phi:\Q]}\log|\disc(E^*_\Phi)|+C_{14}(g).
\end{split}
\end{equation}
\end{thm}

\begin{thm}\label{prop:field of good reduction can be small}
For any $g\in \Z_{\ge 1}$, there exist effectively computable constants $C_{15}(g)>0$, $C_{16}(g)\in \R$, such that for any CM-field $E$ with $[E:\Q]=2g$, for any CM-type $\Phi$ of $E$, there exists a number field $K'$ and a CM abelian variety $(A, i\colon \OO_E\into \End_{K'}(A))$ over $K'$ of CM-type $\Phi$ such that the abelian variety $A$ over $K'$ has everywhere good reduction and 
\begin{equation}\label{equ:bound on disc(K')}
\begin{split}
\frac{1}{[K':\Q]}\log|\disc(K')|\le C_{15}(g)\log|\disc(E)|+C_{16}(g).
\end{split}
\end{equation}
\end{thm}

This will be discussed in detail in \autoref{section:Field of everywhere good reduction of abelian varieties with complex multiplication}. 

In Theorem 6(ii) of \cite{ColmezFaltingsheight}, Colmez has proved that there exist effectively computable absolute constants $\CColone>0$, $\CColtwo\in \R$ such that for any CM-field $E$ of degree $[E:\Q]=2g$ and any CM-type $\Phi$ of $E$ such that the following hold:
\begin{enumerate}
\item  $(E, \Phi)$ satisfies the Colmez conjecture,
\item  For any irreducible Artin character $\chi$ such that $\Colm(\chi)\ne 0$, the Artin conjecture for $\chi$ holds (i.e. the Artin $L$-function $L(s, \chi, \Q)$ is holomorphic everywhere except possibly for a simple pole at $s=1$),
\item  For any irreducible Artin character $\chi$ such that $\Colm(\chi)\ne 0$, the Artin $L$-function $L(s, \chi, \Q)$ has no zeros on the ball of radius $\frac{1}{4}$ centered at $0$,
\end{enumerate}
we have
\begin{equation*}
\begin{split}
\hFaltE \le \CColone \cdot \mu_{(E, \Phi)}+g \CColtwo .
\end{split}
\end{equation*}

The proofs of \autoref{thm:1} and \autoref{thm:3} show that that we can actually remove the second hypothesis that the Artin $L$-functions involved satisfy the Artin conjecture. Moreover, in the remark after Theorem 6 of \cite{ColmezFaltingsheight}, Colmez asked whether it is possible to remove the third hypothesis that the Artin $L$-functions involved have no zeros on the ball of radius $\frac{1}{4}$ centered at $0$ and making use of ``no Siegel zeros" instead. The proofs of \autoref{thm:1} and \autoref{thm:3} is more or less a positive answer to this question. 

\subsection*{Acknowledgements}

The author is deeply grateful to Professor Wei Zhang for suggesting this problem to the author, supervising the author on this project, and teaching, mentoring and guiding the author all along. For many times the author would have been in distress had it not been for the kind and generous help of Professor Zhang. The author also thanks the Undergraduate Research Opportunities Program of MIT for providing this opportunity for the author to do undergraduate research.

\section{The Faltings height}\label{subsection:The Faltings height of an abelian variety}

Let $A$ be a dimension-$g$ abelian variety defined over a number field $K$. Let $\pi\colon \A \to \Spec(\OO_K)$ be the N\eup ron model of $A$, and take $\omega$ to be any global section of $\LL\coloneqq \pi_* \Omega^g_{\A/\Spec\OO_K}$. We define the \textit{unstable Faltings height} of $A$ as follows:
\begin{equation*}
\begin{split}
\hfalt(A/K)\coloneqq \frac{1}{[K:\Q]}\vast(&\log \#\biggl(H^0(\Spec \OO_K, \LL)/(\OO_K\cdot \omega)\biggr)\\
&-\frac{1}{2}\sum_{\sigma\colon K\to \C}\log\Biggl(\frac{1}{(2\pi)^g}\biggl| \int_{A(\C)} \sigma(\omega \wedge \overline{\omega}) \biggr|\Biggr)\vast).
\end{split}
\end{equation*}
This definition is independent of the choice of $\omega\in H^0(\Spec \OO_K, \LL)$. 

We define the \textit{stable Faltings height} of $A$ to be 
\[
\hst(A)\coloneqq \hfalt(A_{K'}/K'),
\]
where $K'$ is a finite extension of $K$ such that $A_{K'}/K'$ has everywhere semistable reduction. This definition does not depend on the choice of the finite extension $K'/K$. Unlike the unstable Faltings height, the stable Faltings height does not depend on the field of definition of the abelian variety. 

The following is a theorem of Bost (\cite{Bost}).

\begin{thm}\label{thm:Bost}
There exists an effectively computable absolute constant $\Clower>0$ such that for any dimension-$g$ abelian variety $A$ over a number field, we have
\[
\hst(A)\ge -g \Clower. 
\]
\end{thm}

As we have mentioned in \autoref{section:Introduction}, it is proved by Colmez in \cite{Colmez} that for any CM-field $E$ and any CM-type $\Phi$ of $E$, if $(A_1, i_1\colon \OO_E\into \End_{K_1}(A_1))$ and $(A_2, i_2\colon \OO_E\into \End_{K_1}(A_2))$ are CM abelian varieties over number fields $K_1$ and $K_2$, both with CM-type $\Phi$, then 
\[
\hst(A_1)=\hst(A_2). 
\]
We denote this stable Faltings height as $\hFaltE$. 

\section{The Colmez conjecture revisited}

Throughout this section $g$ is an arbitrary positive integer, $E$ is an arbitrary CM-field of degree $[E:\Q]=2g$ and $\Phi$ is an arbitrary CM-type of $E$. We denote as $E^*_\Phi$ the reflex field of $(E, \Phi)$. 

Let $A_{(E, \Phi)}^0$ be the function from $\Gal(\Qbar/\Q)$ to $\C$ defined as in \autoref{section:Introduction}. Then since $\Stab(\Phi)\subset \Gal(\Qbar/\Q)$ is equal to $\Gal(\Qbar/E^*_\Phi)$, we have: $A_{(E, \Phi)}^0$ factors as
\[
\Gal(\Qbar/\Q)\surto \Gal(\NE/\Q) \to \C,
\]
where we denote as $\NE$ the Galois closure of the extension $E^*_\Phi/\Q$. For the following we view the function $A_{(E, \Phi)}^0$ as a (class) function from $\Gal(\NE/\Q)$ to $\C$. For any irreducible Artin character $\chi$ such that $\Colm(\chi)\ne 0$, $\chi$ also factors as 
\[
\Gal(\Qbar/\Q)\surto \Gal(\NE /\Q) \to \C,
\]
and so we also view $\chi$ as a character of $\Gal(\NE/\Q)$.

We fix an embedding $\Qbar \into \C$ and let $\iota$ be the element of $\Gal(\Qbar/\Q)$ induced by complex conjugation. Let $\chi$ be an irreducible Artin character. We say that $\chi$ is \textit{odd} if $\chi(\iota)=-\chi(1)$. Some computations show that for the trivial character $\chi=\one$, we have $\Colm(\one)=\frac{1}{2}g$; and for any nontrivial irreducible Artin character $\chi$, $\Colm(\chi)=0$ unless $\chi$ is odd. 

Therefore, we have
\begin{equation}\label{equ:Z3}
A_{(E, \Phi)}^0 =\frac{1}{2}g \cdot \one+ \sum_{\substack{\chi \in \Irr(\Gal(\NE/\Q)) \\ \chi\ne \one\\ \chi \mbox{ \rm{odd}}}} \Colm(\chi) \chi.
\end{equation}
This implies that for any irreducible Artin character $\chi$ such that $m_{(E, \Phi)}(\chi)\ne 0$, the Artin $L$-function $L(s, \chi, \Q)$ is defined and nonzero at $s=0$.

Let $Z_{(E. \Phi)}$ and $\mu_{(E, \Phi)}$ be as in \autoref{section:Introduction}. Since
\[
\frac{\zeta'_\Q(0)}{\zeta_\Q(0)}=\log(2\pi)
\]
and
\[
\f(\one, \Q)=0,
\]
we can deduce that
\begin{equation}\label{equ:Z1}
\ColZ = \sum_{\substack{\chi \in \Irr(\Gal(\NE/\Q)) \\ \chi\ne \one\\ \chi \mbox{ \rm{odd}}}} \Colm(\chi) \frac{L'(0, \chi, \Q)}{L(0, \chi, \Q)},
\end{equation}
and
\begin{equation}\label{equ:Z5}
\Colmu = \sum_{\substack{\chi \in \Irr(\Gal(\NE/\Q)) \\ \chi\ne \one\\ \chi \mbox{ \rm{odd}}}} \Colm(\chi) \log \f(\chi, \Q).
\end{equation}

\section{The zero of the Artin $L$-function near $1$}

\subsection{Relation between the zero near $1$ and the logarithmic derivative at $0$ of the Artin $L$-function}

Throughout this subsection $g$ is an arbitrary positive integer, $E$ is an arbitrary CM-field of degree $[E:\Q]=2g$ and $\Phi$ is an arbitrary CM-type of $E$. We denote as $E^*_\Phi$ the reflex field of $(E, \Phi)$. We denote as $\NE$ the Galois closure of the extension $E^*_\Phi/\Q$. 

By Chapter 2, Section 5 of \cite{Murty}, for any nontrivial irreducible character $\chi$ of $\Gal(\NE/\Q)$, the functions $\frac{\zeta_{\NE}(s)}{L(s, \chi, \Q)}$ and $\zeta_{\NE}(s)L(s, \chi, \Q)$ are both holomorphic except for a simple pole at $s=1$. By Lemma 3 of \cite{Stark}, for any number field $K$ such that $K\ne \Q$, the function $\zeta_{K}(s)$ has at most one zero in the region 
\[
1-\frac{1}{4\log|\disc(K)|}\le \re(s)<1, |\im(s)|\le \frac{1}{4\log|\disc(K)|}. 
\]
If such a zero exists, it is real and simple. 

Therefore, the function $L(s, \chi, \Q)$ has at most one zero in the region 
\[
1-\frac{1}{4\log|\disc(\NE)|}\le \re(s)<1, |\im(s)|\le \frac{1}{4\log|\disc(\NE)|}. 
\]
If such a zero exists, it is real and simple.

\begin{prop}\label{prop:1}
Let $\chi$ be a nontrivial odd irreducible character of $\Gal(\NE/\Q)$. Denote as $\be_0$ the (necessarily real and simple) zero of $L(s, \chi, \Q)$ in the region 
\[
1-\frac{1}{4\log|\disc(\NE)|}\le \re(s)<1, |\im(s)|\le \frac{1}{4\log|\disc(\NE)|} 
\]
(if it exists). 

Let $\delta_{\chi}$ be $1$ if $\be_0$ exists, and let $\delta_{\chi}$ be $0$ otherwise. We have
\begin{equation*}
\begin{split}
&-\biggl(\frac{L'(0, \chi, \Q)}{L(0, \chi, \Q)}+\frac{L'(0, \chibar, \Q)}{L(0, \chibar, \Q)}+ \frac{2\del_{\chi}}{1-\be_0}\biggr)\\
> & 
- 75 \log|\disc(\NE)|+
\Biggl(\log(\frac{\f(\chi, \Q)}{\pi^{\chi(1)}}) +\chi(1)\frac{\Ga'\bigl(\frac{1}{2} \bigr)}{\Ga\bigl(\frac{1}{2}\bigr)}\Biggr),   
\end{split}    
\end{equation*}
and
\begin{equation}\label{equ:Z4}
\begin{split}
&-\biggl(\frac{L'(0, \chi, \Q)}{L(0, \chi, \Q)}+\frac{L'(0, \chibar, \Q)}{L(0, \chibar, \Q)}+ \frac{2\del_{\chi}}{1-\be_0}\biggr)\\
< & 75 \log|\disc(\NE)|+
\Biggl(\log(\frac{\f(\chi, \Q)}{\pi^{\chi(1)}}) +\chi(1)\frac{\Ga'\bigl(\frac{1}{2} \bigr)}{\Ga\bigl(\frac{1}{2}\bigr)}\Biggr). 
\end{split}    
\end{equation}
\end{prop}
\begin{proof}
We define the function $\La(s, \chi, \Q)$ to be
\[
\La(s, \chi, \Q)\coloneqq (\f(\chi, \Q))^{s/2} \Biggl( \pi^{-(s+1)/2}\Ga\bigl(\frac{s+1}{2}\bigr)\Biggr)^{\chi(1)} L(s, \chi, \Q).
\]

We have the functional equation
\[
\La(s, \chi, \Q)=W(\chi)\La(1-s, \chibar, \Q)
\]
for some $W(\chi)\in \C$ with absolute value $1$. 

We define the function $\xi_{\NE}$ to be 
\[
\xi_{\NE}(s)\coloneqq s(s-1)|\disc(\NE)|^{s/2}\Biggl(2(2\pi)^{-s}\Ga(s)\Biggr)^{[\NE:\Q]/2}\zeta_{\NE}(s).
\]

We have the functional equation 
\[
\xi_{\NE}(s)=\xi_{\NE}(1-s). 
\]

First consider the function $f_1(s)\coloneqq (\xi_{\NE}(s))^2\La(s, \chi, \Q)\La(s, \chibar, \Q)$. It is entire and satisfies the functional equation 
\[
f_1(s)=f_1(1-s).
\]
Since $f_1(s)$ is real for $s$ real, for any $\rho\in \C$ the order of the zero of $f_1(s)$ at $s=\rho$ is equal to that at $s=\rhobar$. Moreover, all zeros of $f_1(s)$ lie in the critical strip $0<\re(s)<1$. Therefore, by logarithmically differentiating the Hadamard product formula for $f_1(s)$ at $s=1$ we get
\[
\sum_{\rho\colon f_1(\rho)=0} \frac{1}{2}\biggl(\frac{1}{1-\rho}+\frac{1}{1-\rhobar}\biggr)=\frac{f'_1(1)}{f_1(1)}, 
\]
i.e. 
\begin{equation}\label{equ:Z2}
\begin{split}
&\sum_{\rho \colon f_1(\rho)=0} \frac{1}{2}\biggl(\frac{1}{1-\rho}+\frac{1}{1-\rhobar}\biggr)\\
=&2\frac{\xi'_{\NE}(1)}{\xi_{\NE}(1)}+\frac{\La'(1, \chi, \Q)}{\La(1, \chi, \Q)}+\frac{\La'(1, \chibar, \Q)}{\La(1, \chibar, \Q)}.
\end{split}
\end{equation}

Let $\delta_{\NE}$ be $1$ if $\be_0$ is a zero of $\zeta_{\NE}(s)$, and let $\delta_{\NE}$ be $0$ otherwise. Then $\delta_{\NE}-\delta_{\chi}$ is equal to $0$ or $1$ and the order of the zero of $f_1(s)$ at $s=\be_0$ is equal to $2\delta_{\NE}+2\delta_{\chi}$. Thus, we have
\begin{equation}\label{equ:Y1}
\begin{split}
&\sum_{\substack{\rho\colon f_1(\rho)=0 \\ \rho\ne \be_0}} \frac{1}{2}\biggl(\frac{1}{1-\rho}+\frac{1}{1-\rhobar}\biggr)\\
=&\Biggl(2\frac{\xi'_{\NE}(1)}{\xi_{\NE}(1)}-\frac{2\delta_{\NE}}{1-\be_0}\Biggr)+\Biggl(\frac{\La'(1, \chi, \Q)}{\La(1, \chi, \Q)}+\frac{\La'(1, \chibar, \Q)}{\La(1, \chibar, \Q)}-\frac{2\delta_{\chi}}{1-\be_0}\Biggr).
\end{split}
\end{equation}

Since the function $\frac{(\zeta_{\NE}(s))^2}{L(s, \chi, \Q)L(s, \chibar, \Q)}$ is holomorphic on $0<\re(s)<1$, for any $\rho\in \C$ such that $0<\re(\rho)<1$, the order of the zero at $s=\rho$ of the function $f_1(s)$ is less than or equal to $4$ times the order of the zero at $s=\rho$ of the function $\zeta_{\NE}(s)$. In view of the fact that all zeros of $f_1(s)$ lie in the critical strip $0<\re(s)<1$, we have

\begin{equation}\label{equ:Y3}
0\le \sum_{\substack{\rho \colon f_1(\rho)=0 \\ \rho\ne \be_0}} \frac{1}{2}\biggl(\frac{1}{1-\rho}+\frac{1}{1-\rhobar}\biggr)\le 4 \sum_{\substack{\rho\colon \zeta_{\NE}(\rho)=0\\ 0<\re(\rho)<1 \\ \rho\ne \be_0}} \frac{1}{2}\biggl(\frac{1}{1-\rho}+\frac{1}{1-\rhobar}\biggr).
\end{equation}

Then consider the function $f_2(s)\coloneqq \frac{(\xi_{\NE}(s))^2}{\La(s, \chi, \Q)\La(s, \chibar, \Q)}$. Similar to the case of $f_1$, we have
\begin{equation}\label{equ:Y2}
\begin{split}
&\sum_{\substack{\rho \colon f_2(\rho)=0 \\ \rho\ne \be_0}} \frac{1}{2}\biggl(\frac{1}{1-\rho}+\frac{1}{1-\rhobar}\biggr)\\
=&\Biggl(2\frac{\xi'_{\NE}(1)}{\xi_{\NE}(1)}-\frac{2\delta_{\NE}}{1-\be_0}\Biggr)-\Biggl(\frac{\La'(1, \chi, \Q)}{\La(1, \chi, \Q)}+\frac{\La'(1, \chibar, \Q)}{\La(1, \chibar, \Q)}-\frac{2\delta_{\chi}}{1-\be_0}\Biggr),
\end{split}
\end{equation}

and

\begin{equation}\label{equ:Y4}
0\le \sum_{\substack{\rho \colon f_2(\rho)=0 \\ \rho\ne \be_0}} \frac{1}{2}\biggl(\frac{1}{1-\rho}+\frac{1}{1-\rhobar}\biggr)\le 4 \sum_{\substack{\rho\colon \zeta_{\NE}(\rho)=0\\ 0<\re(\rho)<1 \\ \rho\ne \be_0}} \frac{1}{2}\biggl(\frac{1}{1-\rho}+\frac{1}{1-\rhobar}\biggr).
\end{equation}

By logarithmically differentiating the functional equation of $\La(s, \chi, \Q)\La(s, \chibar, \Q)$ at $s=1$, we have
\[
\frac{\La'(0, \chi, \Q)}{\La(0, \chi, \Q)}+\frac{\La'(0, \chibar, \Q)}{\La(0, \chibar, \Q)}=-\Biggl(\frac{\La'(1, \chi, \Q)}{\La(1, \chi, \Q)}+\frac{\La'(1, \chibar, \Q)}{\La(1, \chibar, \Q)}\Biggr).
\] 

The result then follows from subtracting Equation (\ref{equ:Y1}) by Equation (\ref{equ:Y2}) and the following \autoref{prop:sum of zeros of zeta_K without Siegel zero is small}. 
\end{proof}

\begin{lem}\label{prop:sum of zeros of zeta_K without Siegel zero is small}	
Let $K$ be a number field such that $K\ne \Q$. Denote as $\be_0$ the (necessarily real and simple) zero of $\zeta_K(s)$ 
in the region
\[
1-\frac{1}{4\log|\disc(K)|}\le \re(s)<1, |\im(s)|\le \frac{1}{4\log|\disc(K)|}
\]
(if it exists). Then we have
\begin{equation*}
\begin{split}
0\le \SumzetaKnotSiegel \frac{1}{2} \onerhoonerhobar <\frac{75}{2} \log|\disc(K)|.
\end{split}    
\end{equation*}
\end{lem}
\begin{proof}
For any $\rho\in \C$ such that $\re(\rho)<1-\frac{1}{4\log|\disc(K)|}$, we have
\begin{equation*}
\begin{split}
0<&\frac{1}{1-\rho}+\frac{1}{1-\rhobar}< 25 \cdot \discKrhodiscKrhobar.
\end{split}    
\end{equation*}

For any $\rho\in \C$ such that $1-\frac{1}{4\log|\disc(K)|}\le \re(\rho)<1$ and $|\im(\rho)|> \frac{1}{4\log|\disc(K)|}$, we have
\begin{equation*}
\begin{split}
0<\frac{1}{1-\rho}+\frac{1}{1-\rhobar}< 5 \cdot \discKrhodiscKrhobar.
\end{split}    
\end{equation*}

Therefore, we have
\begin{equation*}
\begin{split}
0\le &\SumzetaKnotSiegel \frac{1}{2} \onerhoonerhobar \\
<&25\SumzetaKnotSiegel \frac{1}{2}\discKrhodiscKrhobar \\
\le & 25\SumzetaK \frac{1}{2}\discKrhodiscKrhobar.
\end{split}    
\end{equation*}

By the proof of Lemma 3 of \cite{Stark}, we have
\begin{equation}\label{equ:bound on sum over zeros of zeta_K}
0\le \SumzetaK \frac{1}{2} \biggl(\frac{1}{s-\rho}+\frac{1}{s-\rhobar} \biggr) <\frac{1}{s-1}+\frac{1}{2}\log|\disc(K)|,
\end{equation}
for any $s$ real with $1<s<2$. 
	
Taking $s=1+\frac{1}{\log|\disc(K)|}$ in Equation (\ref{equ:bound on sum over zeros of zeta_K}), we get:
\begin{equation*}
\begin{split}
0 \le & \SumzetaK \frac{1}{2} \discKrhodiscKrhobar\\
<& \frac{3}{2}\log|\disc(K)|.
\end{split}    
\end{equation*}
	
Therefore, we have
\begin{equation*}
\begin{split}
0\le \SumzetaKnotSiegel \frac{1}{2}\onerhoonerhobar
< \frac{75}{2}\log|\disc(K)|.
\end{split}    
\end{equation*}
\end{proof}

\begin{cor}\label{prop:bound on Colmez height assuming no Stark zero}
Let $c$ be a real number such that $0<c\le\frac{1}{4}$. Suppose that for any nontrivial odd irreducible character $\chi$ of $\Gal(\NE/\Q)$ such that $\Colm(\chi)\ne 0$, there is no zero of $L(s, \chi, \Q)$ in the region
\[
1-\frac{c}{\log|\disc(\NE)|}\le \re(s)<1, |\im(s)|\le \frac{1}{4\log|\disc(\NE)|},
\]
then we have
\begin{equation}\label{equ:L1}
\begin{split}
-\ColZ-\frac{1}{2}\Colmu<& \frac{1}{4}g  \cdot (75+2c')(2g)!\log|\disc(E^*_\Phi)|+\frac{1}{4} g \biggl(\frac{\Ga'\bigl(\frac{1}{2} \bigr)}{\Ga\bigl(\frac{1}{2}\bigr)}-\log(\pi)\biggr),
\end{split}
\end{equation}
where $c'$ is defined to be $\frac{1}{c}$ if $c<\frac{1}{4}$, and $0$ if $c=\frac{1}{4}$. 
\end{cor}
\begin{proof}
By Lemma 2 and Section 2 of \cite{ColmezFaltingsheight}, for any $\chi\in \Irr(\Gal(\NE/\Q))$, $\Colm(\chi)$ is a non-negative real number and $\Colm(\chi)=\Colm(\chibar)$. Hence, we have
\begin{equation}\label{equ:M1}
\begin{split}
&-\ColZ-\frac{1}{2}\Colmu\\
=&\sum_{\substack{\chi \in \Irr(\Gal(\NE/\Q)) \\ \chi\ne \one \\ \chi \mbox{ odd}}} \Colm(\chi) \biggl(-\frac{L'(0, \chi, \Q)}{L(0, \chi, \Q)}-\frac{1}{2} \log(\f(\chi, \Q))\biggr)\\
=&\frac{1}{2}\sum_{\substack{\chi \in \Irr(\Gal(\NE/\Q)) \\ \chi\ne \one\\ \chi \mbox{ odd}}} \Colm(\chi)  \Biggl(-\biggl(\frac{L'(0, \chi, \Q)}{L(0, \chi, \Q)}+\frac{L'(0, \chibar, \Q)}{L(0, \chibar, \Q)}\biggr)-\log(\f(\chi, \Q))\Biggr)\\
<& \frac{1}{2}\sum_{\substack{\chi \in \Irr(\Gal(\NE/\Q)) \\ \chi\ne \one\\ \chi \mbox{ odd}}} \Colm(\chi)  \Biggl((75+2c') \log|\disc(\NE)|+
\chi(1)\biggl(\frac{\Ga'\bigl(\frac{1}{2} \bigr)}{\Ga\bigl(\frac{1}{2}\bigr)}-\log(\pi)\biggr) \Biggr),
\end{split}
\end{equation}
by Equation (\ref{equ:Z4}). 

By the definition of $\ColA$ we have \[\sum_{\substack{\chi \in \Irr(\Gal(\NE/\Q)) \\ \chi\ne \one \\ \chi \mbox{ odd}}} \Colm(\chi)\chi(1)=\ColA(1)-\frac{1}{2}g=\frac{1}{2}g.\] Since for any $\chi\in \Irr(\Gal(\NE/\Q))$, $\Colm(\chi)$ is a non-negative real number and $\chi(1)\ge 1$, we have \[\sum_{\substack{\chi \in \Irr(\Gal(\NE/\Q)) \\ \chi\ne \one \\ \chi \mbox{ odd}}} \Colm(\chi)\le \frac{1}{2}g.\] 

By the following Equation (\ref{equ:X3}) we have
\begin{equation*}
\begin{split}
\frac{1}{[\NE:\Q]}\log|\disc(\NE)|\le \log|\disc(E^*_\Phi)|. 
\end{split}
\end{equation*}

The reflex field $E^*_\Phi$ is contained in the Galois closure $\widetilde{E}$ of the extension $E/\Q$, and so $\NE$ is also contained in $\widetilde{E}$. Thus, we have $[\NE:\Q]\le (2g)!$. 

Hence, we get our claim. 
\end{proof}

\begin{lem}\label{lem:disc}
Let $K_1$ and $K_2$ be number fields. Let $K_1K_2$ be the compositum of $K_1$ and $K_2$. Then we have
\begin{equation}\label{equ:X1}
|\disc(K_1K_2)|^{1/[K_1K_2:\Q]}\le |\disc(K_1)|^{1/[K_1:\Q]}|\disc(K_2)|^{1/[K_2:\Q]},
\end{equation}
and
\begin{equation}\label{equ:X2}
|\disc(K_1K_2)|\le |\disc(K_1)|^{[K_2:\Q]}|\disc(K_2)|^{[K_1:\Q]}.
\end{equation}

In particular, let $K$ be a number field and let $\widetilde{K}$ be the Galois closure of the extension $K/\Q$. Then 
\begin{equation}\label{equ:X3}
|\disc(\widetilde{K})|^{1/[\widetilde{K}:\Q]}\le |\disc(K)|,
\end{equation}
and
\begin{equation}\label{equ:X4}
|\disc(\widetilde{K})|\le |\disc(K)|^{[K:\Q]!}.
\end{equation}
\end{lem}
\begin{proof}
This is Lemma 7 of \cite{Stark}. 
\end{proof}

\subsection{Sufficient conditions for the nonexistence of the zero near $1$ of the Artin $L$-function}

By Theorem 3 of \cite{Stark}, we have the following theorem. 

\begin{thm}\label{prop:Theorem 3 of Stark}
Let $L/K$ be a finite Galois extension of number fields. Let $s_0\in \C$ be a simple zero of $\zeta_L(s)$. 

(1) For any irreducible character $\chi$ of $\Gal(L/K)$, $L(s, \chi, K)$ is defined at $s=s_0$. There is a (unique) irreducible character $\mathcal{X}_{s_0, L/K}$ of $\Gal(L/K)$ such that for any irreducible character $\chi$ of $\Gal(L/K)$, $L(s_0, \chi, K)=0$ if and only if $\chi=\mathcal{X}_{s_0, L/K}$. $\mathcal{X}_{s_0, L/K}$ is a linear character of $\Gal(L/K)$ (so $\mathcal{X}_{s_0, L/K}$ is a group homomorphism from $\Gal(L/K)$ to $\C^\times$).

(2) There is a (unique) subfield $\mathcal{K}_{s_0, L/K}$ of $L$ containing $K$ such that for any field $K'$ containing $K$ and contained in $L$, $\zeta_{K'}(s_0)=0$ if and only if $K'$ contains $\mathcal{K}_{s_0, L/K}$. The extension $\mathcal{K}_{s_0, L/K}/K$ is cyclic. 

(3) $\mathcal{K}_{s_0, L/K}$ is the fixed field of the kernel of $\mathcal{X}_{s_0, L/K}$.

(4) Suppose further that $s_0$ is real. Then exactly one of the following holds:
\begin{enumerate}
\item $\mathcal{K}_{s_0, L/K}$ is equal to $K$ and $\mathcal{X}_{s_0, L/K}$ is the trivial character. 
\item $\mathcal{K}_{s_0, L/K}$ is quadratic over $K$ and $\mathcal{X}_{s_0, L/K}$ is the group homomorphism from $\Gal(L/K)$ to $\C^\times$ with kernel $\Gal(L/\mathcal{K}_{s_0, L/K})$ and image $\{\pm 1\}$. In particular, $\mathcal{X}_{s_0, L/K}$ is a nontrivial real linear character. 
\end{enumerate}
\end{thm}

For the rest of this subsection $E$ is an arbitrary CM-field and $\Phi$ is an arbitrary CM-type of $E$. We denote as $E^*_\Phi$ the reflex field of $(E, \Phi)$. We denote as $\NE$ the Galois closure of the extension $E^*_\Phi/\Q$. 

\begin{cor}\label{prop:criterion 1 for no Stark zero of Colmez Artin L-function}
Suppose that one (or two, or all) of the following conditions hold:
\begin{enumerate}
\item The Galois closure $\widetilde{E}$ of the extension $E/\Q$ does not contain any complex quadratic subfields.
\item $\NE$ does not contain any complex quadratic subfields.
\item There does not exist a nontrivial irreducible real linear character $\chi$ of $\Gal(\NE/\Q)$ such that $\Colm(\chi)\ne 0$ and the homomorphism $\chi$ from $\Gal(\NE/\Q)$ to $\C^\times$ has image $\{\pm1\}$ and kernel $\Gal(\NE/K)$ for some complex quadratic subfield $K$ of $\NE$. 
\end{enumerate}
(Note that Condition 1 implies Condition 2 since $E^*_\Phi\subset \widetilde{E}$, and Condition 2 implies Condition 3.)

Then for any nontrivial odd irreducible character $\chi$ of $\Gal(\NE/\Q)$, there is no zero of $L(s, \chi, \Q)$ in the region
\[
1-\frac{1}{4\log|\disc(\NE)|}\le \re(s)<1, |\im(s)|\le \frac{1}{4\log|\disc(\NE)|}.
\]
\end{cor}
\begin{proof}
Let $\chi$ be a nontrivial odd irreducible character of $\Gal(\NE/\Q)$ such that such a zero exists. Denote this zero as $\be_0$. Then $\beta_0$ must be real and $\be_0$ is also a simple zero of $\zeta_{\NE}(s)$. 

Therefore, by \autoref{prop:Theorem 3 of Stark}, $\chi$ is a real linear character of $\Gal(\NE/\Q)$, and the homomorphism $\chi$ from $\Gal(\NE/\Q)$ to $\C^\times$ has image $\{\pm 1\}$ and kernel $\Gal(\NE/K)$ for some quadratic subfield $K$ of $\NE$. 

Since $\chi$ is an odd character, we have $\chi(\iota)=-\chi(1)$, where $\iota$ is the element in $\Gal(\NE/\Q)$ induced by complex conjugation, and so $K/\Q$ must be a \textit{complex} quadratic extension.

Therefore, our claim follows.  
\end{proof}

Since the compositum of two CM-fields is also a CM-field, the Galois closure of a CM-field (viewed as an extension over $\Q$) is also a CM-field. We know that the reflex field $E^*_\Phi$ of $(E, \Phi)$ is a CM-field. Therefore, $\NE$ is also a CM-field. We denote as $\NEp$ the maximal totally real subfield of $\NE$. 

\begin{prop}\label{prop:criterion 2 for no Stark zero of Colmez Artin L-function}
Let $c$ be a real number such that $0<c\le \frac{1}{4}$. Suppose that the function $L(s, \chi_{\NE/\NEp})=\frac{\zeta_{\NE}}{\zeta_{\NEp}}$ has no zero in the region 
\[
1-\frac{c}{\log|\disc(\NE)|} \le \re(s)<1, |\im(s)|\le \frac{1}{4 \log|\disc(\NE)|}.
\]
Then for any nontrivial odd irreducible character $\chi$ of $\Gal(\NE/\Q)$, there is no zero of $L(s, \chi, \Q)$ in the above region either. 
\end{prop}
\begin{proof}
Let $\chi$ be a nontrivial odd irreducible character of $\Gal(\NE/\Q)$ such that such a zero exists. Denote this zero as $\be_0$. Then $\beta_0$ must be real and $\be_0$ is also a simple zero of $\zeta_{\NE}(s)$. By our assumption on $L(s, \chi_{\NE/\NEp})$, $\be_0$ cannot be a zero of $L(s, \chi_{\NE/\NEp})$. Therefore, $\be_0$ is a zero of $\zeta_{\NEp}(s)$. Therefore, the field $\mathcal{K}_{\be_0, \NE/\Q}$ in \autoref{prop:Theorem 3 of Stark} must be contained in the field $\NEp$, and so $\mathcal{K}_{\be_0, \NE/\Q}$ is a \textit{real} quadratic field. By \autoref{prop:Theorem 3 of Stark}, since $L(\be_0, \chi, \Q)=0$, $\chi$ is a group homomorphism from $\Gal(\NE/\Q)$ to $\C^\times$ with kernel $\Gal(\NE/\mathcal{K}_{\be_0, \NE/\Q})$, and so $\chi(\iota)=\chi(1)=1$, where $\iota$ is the element in $\Gal(\NE/\Q)$ induced by complex conjugation. This is a contradiction since the character $\chi$ is assumed to be odd. 
\end{proof}

\subsection{Proofs of \autoref{thm:1} and \autoref{thm:3}}

\begin{proof}[Proof of \autoref{thm:1}]
Let $g$ be a positive integer. Let $E$ be a CM-field with maximal totally real subfield $F$ of degree $[F:\Q]=g$. Let $(A, i\colon \OO_E\into \End_K(A))$ be a CM abelian variety over a number field $K$ and let $\Phi$ be the CM-type of $(A, i)$. Then the field $K$ contains the reflex field $E^*_\Phi$. Thus, we have
\begin{equation*}
\begin{split}
\frac{1}{[K:\Q]}\log|\disc(K)|&\ge \frac{1}{[E^*_\Phi:\Q]}\log|\disc(E^*_\Phi)|\\
&\ge \frac{1}{(2g)!}\log|\disc(E^*_\Phi)|,
\end{split}
\end{equation*}
where the last inequality follows from the fact that the reflex field $E^*_\Phi$ is contained in the Galois closure $\widetilde{E}$ of the extension $E/\Q$.

By Lemma 8 and Lemma 9 of \cite{Stark}, suppose that there is a (necessarily real and simple) zero $\be_0$ of $L(s, \chi_{\NE/\NEp})$ in the range
\[
1-\frac{1}{4\log|\disc(\NE)|}\le \re(s)<1, |\im(s)|\le \frac{1}{4\log|\disc(\NE)|},
\]
then there exists a complex quadratic subfield $K$ of $\NE$ such that $\zeta_K(\be_0)=0$ also. Since the Riemann zeta function $\zeta_\Q(s)$ has no real zeros in the range $0<s<1$, this means that $\beta_0$ is a zero of the function $L(s, \chi_{K/\Q})=\frac{\zeta_{K}(s)}{\zeta_\Q(s)}$. Since $K$ is contained in $\NE$, we have $|\disc(\NE)| \ge |\disc(K)|$. Therefore, $\be_0$ is a Siegel zero of $L(s, \chi_{K/\Q})$.

The result then follows from \autoref{prop:criterion 2 for no Stark zero of Colmez Artin L-function} and \autoref{prop:bound on Colmez height assuming no Stark zero}.
\end{proof}

It is proved by Colmez (\cite{Colmez}) and Obus (\cite{Obus}) that the Colmez conjecture is true when the CM-field is abelian:  

\begin{thm}\label{thm:Colmez conjecture is true when the CM-field is abelian}
Let $E$ be a CM-field such that the extension $E/\Q$ is Galois with abelian Galois group. Then we have
\[
\hFaltE=-\ColZ-\frac{1}{2}\Colmu
\]
for any CM-type $\Phi$ of $E$.
\end{thm}

As a corollary, we can prove an unconditional analogue of \autoref{thm:1}.

\begin{proof}[Proof of \autoref{thm:3}]
Similar to the above proof of \autoref{thm:1}, the statement follows from the above-mentioned Lemma 8 and Lemma 9 of \cite{Stark}, \autoref{prop:criterion 1 for no Stark zero of Colmez Artin L-function}, \autoref{prop:bound on Colmez height assuming no Stark zero}, and \autoref{thm:Colmez conjecture is true when the CM-field is abelian}.
\end{proof}

\section{The (proved) averaged Colmez conjecture}

Although the formula $-\ColZ-\frac{1}{2}\Colmu$ in the Colmez conjecture appears very complicated, the average over all CM-types $\Phi$ of a CM-field $E$ is much simpler: As is conjectured in Page 634 of \cite{Colmez} and proved in \cite{YuanZhang} and \cite{AGHM}, we have the following proposition.

\begin{prop}\label{prop:average Colmez height}
Let $E$ be a CM-field with maximal totally real subfield $F$. Then we have
\[
\frac{1}{2^{[F:\Q]}}\sum_{\Phi} \biggl(-\ColZ-\frac{1}{2}\Colmu \biggr)=-\frac{1}{2}\frac{L'(0, \chi_{E/F})}{L(0, \chi_{E/F})}-\frac{1}{4}\log(|\disc(E)|/|\disc(F)|),
\]
where the sum on the left-hand-side is over all CM-types $\Phi$ of $E$.

In other words, the Colmez conjecture implies the \nameref{thm:averaged Colmez conjecture} stated below. 
\end{prop}

\begin{thm}[(Proved) averaged Colmez conjecture]\label{thm:averaged Colmez conjecture}
Let $E$ be a CM-field with maximal totally real subfield $F$. Then we have
\begin{equation}\label{equ:averaged Colmez conjecture}
\frac{1}{2^{[F:\Q]}}\sum_\Phi \hFaltE=-\frac{1}{2}\frac{L'(0, \chi_{E/F})}{L(0, \chi_{E/F})}-\frac{1}{4}\log(|\disc(E)|/|\disc(F)|),    
\end{equation}
where the sum on the left-hand-side is over all CM-types $\Phi$ of $E$.
\end{thm}

This is proved independently by Yuan--Zhang \cite{YuanZhang} and Andreatta--Goren--Howard--Madapusi-Pera \cite{AGHM}. 

In the following, we use the proved averaged Colmez conjecture to prove averaged analogues of \autoref{thm:1} and \autoref{thm:3}.

\begin{prop}\label{prop:no Siegel zeros implies stable Faltings height small}
Let $g$ be a positive integer. Suppose that there exists some effectively computable constant $\Czero(g)\in \R_{\ge 4}$ depending only on $g$ such that for any CM-field $E$ with maximal totally real subfield $F$ such that $[F:\Q]=g$, the function $L(s, \chi_{E/F})$ has no zeros in the region 
\[
1-\frac{1}{\Czero(g)\log|\disc(E)|} \le \re(s)<1, |\im(s)|\le \frac{1}{4 \log|\disc(E)|},
\]  
then there exist effectively computable constants $C_9(g)>0$, $C_{10}(g)\in \R$ depending only on $g$ such that 
\[
\hst(A)\le C_9(g)\log|\disc(E)|+C_{10}(g)
\]
for any CM-field $E$ of degree $[E:\Q]=2g$ and for any abelian variety $A$ over a number field with complex multiplication by $\OO_E$. 
\end{prop}
\begin{proof}
Let $E$ be any CM-field with maximal totally real subfield $F$ such that $[F:\Q]=g$. Denote as $\be_0$ the (necessarily real and simple) zero of $L(s, \chi_{E/F})$ in the region 
\[
1-\frac{1}{4\log|\disc(E)|} \le \re(s)<1, |\im(s)|\le \frac{1}{4 \log|\disc(E)|}
\]  
(if it exists).

We define $\delta_{\chi_{E/F}}$ to be $1$ if $\be_0$ exists, and we define $\delta_{\chi_{E/F}}$ to be $0$ otherwise. By an argument similar to the proof of \autoref{prop:1}, we have
\begin{equation*}
\begin{split}
&-\frac{L'(0, \chi_{E/F})}{L(0, \chi_{E/F})} - \frac{\delta_{\chi_{E/F}}}{1-\be_0}\\ 
\ge &\frac{1}{2}\log(|\disc(E)|/|\disc(F)|)+\frac{g}{2}\biggl(\frac{\Ga'\bigl(\frac{1}{2}\bigr)}{\Ga \bigl(\frac{1}{2}\bigr)}-\log(\pi)\biggr),   
\end{split}    
\end{equation*}
and
\begin{equation*}
\begin{split}
&-\frac{L'(0, \chi_{E/F})}{L(0, \chi_{E/F})} - \frac{\delta_{\chi_{E/F}}}{1-\be_0}\\ <&\frac{75}{2}\log|\disc(E)|+\frac{1}{2}\log(|\disc(E)|/|\disc(F)|)+\frac{g}{2}\biggl(\frac{\Ga'\bigl(\frac{1}{2}\bigr)}{\Ga \bigl(\frac{1}{2}\bigr)}-\log(\pi)\biggr).  
\end{split}    
\end{equation*}

By our assumption, we then have
\begin{equation*}
\begin{split}
&-\frac{1}{2}\frac{L'(0, \chi_{E/F})}{L(0, \chi_{E/F})}-\frac{1}{4}\log(|\disc(E)|/|\disc(F)|)\\ <&\frac{1}{2}\Biggl(\frac{75}{2}\log|\disc(E)|+\Czero(g)\log|\disc(E)|+\frac{g}{2}\biggl(\frac{\Ga'\bigl(\frac{1}{2}\bigr)}{\Ga \bigl(\frac{1}{2}\bigr)}-\log(\pi)\biggr)\Biggr).  
\end{split}    
\end{equation*}

Let $(A, i\colon \OO_E\into \End_K(A))$ be any CM abelian variety over a number field $K$. Let $\Phi_0$ be the CM-type of $(A, i)$. By \autoref{thm:averaged Colmez conjecture}, we have
\begin{equation*}
\hst(A)=-\sum_{\Phi\ne \Phi_0} \hFaltE+2^g\Biggl(-\frac{1}{2}\frac{L'(0, \chi_{E/F})}{L(0, \chi_{E/F})}-\frac{1}{4}\log(|\disc(E)|/|\disc(F)|)  \Biggr).
\end{equation*}

Let $\Clower>0$ be as in \autoref{thm:Bost}. Then by \autoref{thm:Bost} we have
\begin{equation*}
\begin{split}
\hst(A)&=-\sum_{\Phi\ne \Phi_0} \hFaltE+2^g\Biggl(-\frac{1}{2}\frac{L'(0, \chi_{E/F})}{L(0, \chi_{E/F})}-\frac{1}{4}\log(|\disc(E)|/|\disc(F)|)  \Biggr)\\
&\le (2^g-1)g\Clower\\
& +2^g\cdot \frac{1}{2}\Biggl(\frac{75}{2}\log|\disc(E)|+\Czero(g)\log|\disc(E)|+\frac{g}{2}\biggl(\frac{\Ga'\bigl(\frac{1}{2}\bigr)}{\Ga \bigl(\frac{1}{2}\bigr)}-\log(\pi)\biggr)\Biggr).
\end{split}
\end{equation*}
\end{proof}

\begin{prop}\label{prop:stable Faltings height small for CM-fields without complex quadratic subfields}
For any $g\in \Z_{\ge 1}$, there exist constants $C_{11}(g)>0$, $C_{12}(g)\in \R$ depending only on $g$ such that 
\[
\hst(A)\le C_{11}(g)\log|\disc(E)|+C_{12}(g)
\]
for any CM-field $E$ of degree $[E:\Q]=2g$ such that $E$ has no complex quadratic subfields and for any abelian variety $A$ over a number field with complex multiplication by $\OO_E$. 
\end{prop}
\begin{proof}
Let $g$ be a positive integer. Let $E$ be a CM-field with maximal totally real subfield $F$ with degree $[F:\Q]=g$. By Lemma 9 of \cite{Stark}, suppose that there exists a (necessarily real and simple) zero $\be_0$ of $L(s, \chi_{E/F})$ in the range
\[
1-\frac{1}{16g!\log|\disc(E)|}\le \re(s)<1, |\im(s)|\le \frac{1}{4\log|\disc(E)|}, 
\]
then there exists a complex quadratic subfield $K$ of $E$ such that $\zeta_{K}(\be_0)=0$ as well. So if $E$ does not contain any complex quadratic fields, then there is no such zero.

The rest of the proof is similar to that of \autoref{prop:no Siegel zeros implies stable Faltings height small}.
\end{proof}

\begin{lem}\label{lem:compositum of reflex fields and F contains E}
Let $E$ be a CM-field with maximal totally real subfield $F$ of degree $[F:\Q]=g$. Let $\Phi_1, \Phi_2$ be CM-types of $E$ such that $|\Phi_1\cap \Phi_2|=g-1$. Let $\varphi_0$ be the unique element in $\Hom_\Q(F, \R)$ such that the element $\phi_1$ in $\Phi_1$ lying above $\varphi_0$ is not equal to the element $\phi_2$ in $\Phi_2$ lying above $\varphi_0$. We have $\phi_1=\phi_2\circ \iota$, where $\iota$ is the nontrivial element of $\Gal(E/F)$. It is easy to see that the subfield $\phi_1(E)$ of $\C$ is equal to the subfield $\phi_2(E)$ of $\C$. Let $E^*_{\Phi_1}, E^*_{\Phi_2}$ be the reflex fields of $(E, \Phi_1), (E, \Phi_2)$, respectively. Then the compositum of fields $E^*_{\Phi_1}E^*_{\Phi_2}$ contains the field $\phi_1(E)=\phi_2(E)$.
\end{lem}
\begin{proof}
Since $E$ is a totally complex quadratic extension of the totally real field $F$, we can write $E=F[\sqrt{-\ai_E}]$ for some totally positive $\ai_E\in F$, where $\sqrt{-\ai_E}$ is any square root of $-\ai_E$ in $\Qbar$. Thus, $\sum_{\phi\in \Phi_1}\phi(\sqrt{-\ai_E})\in E^*_{\Phi_1}$ and $\sum_{\phi\in \Phi_2}\phi(\sqrt{-\ai_E})\in E^*_{\Phi_2}$. By our assumption on $\Phi_1, \Phi_2$ and $\varphi_0$, we have
\begin{equation*}
\begin{split}
\sum_{\phi\in \Phi_1}\phi(\sqrt{-\ai_E})-\sum_{\phi\in \Phi_2}\phi(\sqrt{-\ai_E})&=\phi_1(\sqrt{-\ai_E})-\phi_2(\sqrt{-\ai_E})\\
&=2\phi_1(\sqrt{-\ai_E})\\
&=-2\phi_2(\sqrt{-\ai_E}).
\end{split}
\end{equation*}

Therefore, the compositum of fields $E^*_{\Phi_1}E^*_{\Phi_2}$ contains the element $\phi_1(\sqrt{-\ai_E})=-\phi_2(\sqrt{-\ai_E})$. 

Let $\ai_F$ be an element of $F$ such that $F=\Q[\ai_F]$. Then similar to above, since
\begin{equation*}
\begin{split}
\sum_{\phi\in \Phi_1}\phi(\ai_F\sqrt{-\ai_E})-\sum_{\phi\in \Phi_2}\phi(\ai_F\sqrt{-\ai_E})&=\phi_1(\ai_F\sqrt{-\ai_E})-\phi_2(\ai_F\sqrt{-\ai_E})\\
&=\varphi_0(\ai_F)\phi_1(\sqrt{-\ai_E})-\varphi_0(\ai_F)\phi_2(\sqrt{-\ai_E})\\
&=2\varphi_0(\ai_F)\phi_1(\sqrt{-\ai_E})\\
&=-2\varphi_0(\ai_F)\phi_2(\sqrt{-\ai_E}),
\end{split}
\end{equation*}
the compositum of fields $E^*_{\Phi_1}E^*_{\Phi_2}$ contains the element \\ $\varphi_0(\ai_F)\phi_1(\sqrt{-\ai_E})=-\varphi_0(\ai_F)\phi_2(\sqrt{-\ai_E})$. 

Combined with above, we have: the compositum of fields $E^*_{\Phi_1}E^*_{\Phi_2}$ contains the element $\varphi_0(\ai_F)$ and the element $\phi_1(\sqrt{-\ai_E})=-\phi_2(\sqrt{-\ai_E})$, and so it contains the field $\phi_1(E)=\phi_2(E)$.
\end{proof}

\begin{rem}
The CM-types $\Phi_1, \Phi_2$ in \autoref{lem:compositum of reflex fields and F contains E} is a pair of ``nearby" CM-types considered in \cite{YuanZhang}. 
\end{rem}

\begin{cor}\label{lem:lower bound on reflex field}
Let $E$ be a CM-field with maximal totally real subfield $F$. Let $\Phi_1, \Phi_2$ be CM-types of $E$ such that $|\Phi_1\cap \Phi_2|=g-1$. Then we have
\begin{equation*}
\begin{split}
|\disc(E^*_{\Phi_1})|^{1/[E^*_{\Phi_1}:\Q]}|\disc(E^*_{\Phi_2})|^{1/[E^*_{\Phi_2}:\Q]}\ge |\disc(E)|^{1/[E:\Q]},
\end{split}
\end{equation*}
where $E^*_{\Phi_1}, E^*_{\Phi_2}$ are the reflex fields of $(E, \Phi_1), (E, \Phi_2)$, respectively.
\end{cor}
\begin{proof}
This follows from \autoref{lem:compositum of reflex fields and F contains E} and Equation (\ref{equ:X1}). 
\end{proof}

\begin{proof}[Proof of \autoref{thm:2}]
This follows from \autoref{lem:lower bound on reflex field}, \autoref{prop:stable Faltings height small for CM-fields without complex quadratic subfields}, and the fact that the field of definition of any CM abelian variety contains the reflex field. 
\end{proof}

\begin{proof}[Proof of \autoref{thm:4}]
This follows from \autoref{lem:lower bound on reflex field}, \autoref{prop:no Siegel zeros implies stable Faltings height small}, and the fact that the field of definition of any CM abelian variety contains the reflex field. 
\end{proof}

\section{Field of everywhere good reduction of CM abelian varieties}\label{section:Field of everywhere good reduction of abelian varieties with complex multiplication}

We know that any abelian variety over a number field with complex multiplication by a CM-field has potential good reduction everywhere. In this section, we show that the logarithm of the root discriminant of the field of everywhere good reduction can be small compared with the logarithm of the discriminant of the CM-field. 

\begin{lem}\label{lem:intersection of field of good reduction is good reduction}
Let $A$ be an abelian variety over a number field $K$. Let $L_1, L_2$ be number fields containing $K$. If the abelian variety $A_{L_1}/L_1$ and the abelian variety $A_{L_2}/L_2$ both have everywhere good reduction, then the abelian variety $A_{L_1\cap L_2}/L_1\cap L_2$ has everywhere good reduction.
\end{lem}
\begin{proof}
This follows from the Neron-Ogg-Shafarevich criterion.
\end{proof}

By part (b) of Corollary 2 to Theorem 2 of \cite{goodreduction}, we have the following theorem:

\begin{thm}\label{thm:K(A[m]) unramified and good reduction}
Let $A$ be an abelian variety over a number field $K$. Let $\p$ be a prime ideal of $\OO_K$. Let $p$ be the characteristic of the residue field $\OO_K/\p$. Suppose that $A/K$ has potential good reduction at $\p$. Let $m$ be any integer $\ge 3$ and prime to $p$. Let $K(A[m])$ be the minimal field of definition of the set of $m$-torsion points $A[m]$ of $A$. The following are equivalent:

(a) The extension $K(A[m])/K$ is unramified at $\p$.

(b) The abelian variety $A/K$ has good reduction at $\p$.
\end{thm}

\begin{cor}\label{prop:unramified field of good reduction}
Let $K$ be a number field. Let $A$ be an abelian variety over $K$ with potential good reduction everywhere. Let $S_{A/K}$ be the set of all prime ideals of $\OO_K$ where the abelian variety $A$ over $K$ does not have good reduction. There exists a finite Galois extension $L/K$, $L/K$ unramified at all primes $\p$ of $\OO_K$ with $\p\notin S_{A/K}$, such that the abelian variety $A_L/L$ has good reduction everywhere. 
\end{cor}
\begin{proof}
We first fix a prime $p_1$ such that the abelian variety $A/K$ has good reduction at every prime ideal $\p_1$ of $\OO_K$ above $p_1$. Let $L_1\coloneqq K(A[p_2])$ be the minimal field of definition of the set of $p_1$-torsion points $A[p_1]$ of $A$. By \autoref{thm:K(A[m]) unramified and good reduction}, we can show that $L_1/K$ is a finite Galois extension unramified at any prime ideal $\p$ of $\OO_K$ such that $\p\notin S_{A/K}$ and the characteristic of the residue field $\OO_K/\p$ is not equal to $p_1$, and the abelian variety $A_{L_1}/L_1$ has everywhere good reduction. 

Next, we fix a prime $p_2$ not equal to $p_1$ such that the abelian variety $A/K$ has good reduction at every prime ideal $\p_2$ of $\OO_K$ above $p_2$. Let $L_2\coloneqq K(A[p_2])$ be the minimal field of definition of the set of $p_2$-torsion points $A[p_2]$ of $A$. Again by \autoref{thm:K(A[m]) unramified and good reduction}, we can show that $L_2/K$ is a finite Galois extension unramified at any prime ideal $\p$ of $\OO_K$ such that $\p\notin S_{A/K}$ and the characteristic of the residue field $\OO_K/\p$ is not equal to $p_2$, and the abelian variety $A_{L_2}/L_2$ has everywhere good reduction.

Now consider the extension $L_1\cap L_2$ of $K$. It is a finite Galois extension unramified at any prime ideal $\p$ of $\OO_K$ such that $\p\notin S_{A/K}$ (since $p_1\ne p_2$). Since the abelian varieties $A_{L_1}/L_1$ and $A_{L_2}/L_2$ both have everywhere good reduction, by \autoref{lem:intersection of field of good reduction is good reduction}, the abelian variety $A_{L_1\cap L_2}/L_1\cap L_2$ also has everywhere good reduction. Taking $L=L_1\cap L_2$, we get our claim.
\end{proof}

The following lemma shows that in terms of unramifiedness, the extension $L/K$ in \autoref{prop:unramified field of good reduction} is the ``best possible". 

\begin{lem}
Let $K$ be a number field. Let $K'/K$ be a finite extension. Let $\p'$ be a prime ideal of $\OO_{K'}$, lying above a prime ideal $\p$ of $\OO_K$. Let $A$ be an abelian variety over $K$. Suppose that the extension $K'/K$ is unramified at $\p'$, and the abelian variety $A_{K'}/K'$ has good reduction at $\p'$, then the abelian variety $A/K$ has good reduction at $\p$.  
\end{lem}
\begin{proof}
This follows from the Neron-Ogg-Shafarevich criterion.
\end{proof}

By Theorem 7 and the remarks before Theorem 7 of \cite{goodreduction}, we have the following theorem:

\begin{thm}\label{thm:degree of field of good reduction}
Let $K$ be a number field. Let $E$ be a CM-field. Let $A$ be an abelian variety over $K$ with complex multiplication by $E$. Let $\mu(E)$ be the group of all roots of unity in $E$. There exists a cyclic extension $C$ of $K$ of degree $[C:K]\le 2\cdot  \#\mu(E)$, such that the abelian variety $A_{C}$ over $C$ has everywhere good reduction.
\end{thm} 

\begin{cor}\label{thm:field of good reduction}
Let $K$ be a number field. Let $E$ be a CM-field. Let $A$ be an abelian variety over $K$ with complex multiplication by $E$. Let $\mu(E)$ be the group of all roots of unity in $E$. Let $S_{A/K}$ be the set of all prime ideals of $\OO_K$ where the abelian variety $A$ over $K$ does not have good reduction. There exists a cyclic extension $K'$ of $K$ of degree $[K':K]\le 2\cdot \#\mu(E)$, $K'/K$ unramified at any prime ideal $\p$ of $\OO_K$ such that $\p \notin S_{A/K}$, such that the abelian variety $A_{K'}$ over $K'$ has everywhere good reduction.
\end{cor}
\begin{proof}
Let $C/K$ be the finite cyclic extension in \autoref{thm:degree of field of good reduction} and let $L/K$ be the finite Galois extension in \autoref{prop:unramified field of good reduction}. Let $K'=C\cap L$. Then $K'/K$ is a cyclic extension of degree $[K':K]\le 2\cdot \#\mu(E)$ and $K'/K$ is unramified at any prime ideal $\p$ of $\OO_K$ such that $\p \notin S_{A/K}$. By \autoref{lem:intersection of field of good reduction is good reduction}, the abelian variety $A_{K'}/K'$ has everywhere good reduction. Hence we get our claim. 
\end{proof}

In order to prove \autoref{prop:assume GRH, field of good reduction can be small} and \autoref{prop:field of good reduction can be small}, we will also need the following theorem, which is a combination of Corollary A.4.6.5, Theorem A.4.5.1 and Remark A.4.5.2 of \cite{CCO}.  

\begin{thm}\label{thm:existence of special abelian variety with CM type}
Let $E$ be a CM-field and let $\Phi$ be a CM-type of $E$. Let $E^*_\Phi$ be the reflex field of $(E, \Phi)$ and let $M$ be the field of moduli for the reflex norm of $(E, \Phi)$ ($M$ is an everywhere unramified finite abelian extension of $E^*_\Phi$). 

There exists a prime $p$ and a CM abelian variety $(A, i\colon \OO_E\into \End_{M}(A))$ over $M$ of CM-type $\Phi$ such that $A$ has good reduction at every prime ideal of $\OO_{M}$ outside $p$.

Moreover, we can choose $p$ such that
\begin{equation}\label{equ:bound on p_Phi}
p\le 2\cdot |\disc(E(\mu_{mp_1p_2\cdots p_s}))|^{\Cprime},
\end{equation}
where $\Cprime$ is an effectively computable absolute constant in $\R_{>0}$, for any positive integer $n$ $\mu_n$ denotes a primitive $n$-th root of unity, $m$ is the order of the group $\mu(E)$ of all roots of unity in $E$, and $p_1, p_2, \cdots, p_s$ are the distinct prime divisors of $m$.

Assuming the Generalized Riemann Hypothesis, the above bound on $p$ can be improved to 
\begin{equation}\label{equ:bound on p_Phi assuming GRH}
p\le 70 \cdot \biggl(\log |\disc(E(\mu_{mp_1p_2\cdots p_s}))|\biggr)^2.
\end{equation}
\end{thm}

\begin{proof}[Proof of \autoref{prop:assume GRH, field of good reduction can be small}]
Assume the Generalized Riemann Hypothesis. 

Let $g$ be a positive integer. Let $E$ be a CM-field such that $[E:\Q]=2g$. Let $\Phi$ be a CM-type of $E$. Let $E^*_\Phi$ be the reflex field of $(E, \Phi)$ and let the field $M$, the abelian variety $A$ over $M$ and the prime $p$ be as in \autoref{thm:existence of special abelian variety with CM type}, such that the upper bound on $p$ is given by Equation (\ref{equ:bound on p_Phi assuming GRH}). 

Denote $K\coloneqq M$. As in \autoref{thm:field of good reduction}, let $S_{A/K}$ be the set of all prime ideals of $\OO_{K}$ where the abelian variety $A$ over $K$ does not have good reduction. By our choice of $A$, for any $\p\in S_{A/K}$, $\p$ lies above the prime $p$. Let $K'$ be the cyclic extension of $K$ in \autoref{thm:field of good reduction} of degree $[K':K]\le 2\cdot  \#\mu(E)$, $K'/K$ unramified at any prime ideal $\p$ of $\OO_K$ such that $\p \notin S_{A/K}$, such that the abelian variety $A_{K'}$ over $K'$ has everywhere good reduction.

Therefore, the extension $K'/K$ is ramified only at the prime ideals $\q$ of $\OO_{K'}$ such that $\q$ lies above the prime $p$. Let $\D_{K'/K}$ be the different of the extension $K'/K$. By Chapter 3, Section 6 of \cite{Localfields}, we have
\[
e_{\q/\p}-1\le \val_\q(\D_{K'/K}) \le e_{\q/\p}-1+\val_\q(e_{\q/\p}),
\]
for any prime ideal $\q$ of $\OO_{K'}$ lying above a prime ideal $\p$ of $\OO_K$, where $e_{\q/\p}$ is the ramification index of $\q | \p$. 

This means that we have 
\[
\val_\q(\D_{K'/K}) \le 2e_{\q/p} e_{\q/\p}\le 2e_{\q/p} [K':K],
\]
where $e_{\q/p}$ is the ramification index of the prime ideal $\q$ of $\OO_{K'}$ lying above the prime ideal $(p)$ of $\Z$. 

Therefore, we have
\[
\D_{K'/K}\Biggl| \prod_{\substack{\q\subset \OO_{K'}\\ \q | p}}\q^{e_{\q/p}\cdot 2\cdot \#\mu(E)},
\]
where the product is over the prime ideals $\q$ of $\OO_{K'}$ above $p$.

Thus, we have
\begin{equation}\label{equ:B1}
\begin{split}
\log\biggl(\Norm_{K'/\Q} (\D_{K'/K})\biggr)
&\le \log\Biggl(\Norm_{K'/\Q} \biggl(\prod_{\substack{\q\subset \OO_{K'}\\ \q | p}}\q^{e_{\q/p}\cdot 2\cdot \#\mu(E)}\biggr)\Biggr)\\
&\le 2\cdot \#\mu(E) \sum_{\substack{\q\subset \OO_{K'}\\ \q | p}}e_{\q/p}\cdot \log\biggl(\Norm_{K'/\Q}(\q)\biggr)\\
&= 2\cdot \#\mu(E) \sum_{\substack{\q\subset \OO_{K'}\\ \q | p}}e_{\q/p}f_{\q/p} \log(p)\\
&=2\cdot \#\mu(E)[K':\Q]\log(p),
\end{split}
\end{equation}
where $f_{\q/p}$ is the residue degree of the prime ideal $\q$ of $\OO_{K'}$ lying above the prime ideal $(p)$ of $\Z$. 

Since the extension $K/E^*_\Phi$ is unramified, the different $\D_{K/E^*_\Phi}$ of the extension $K/E^*_\Phi$ is equal to the unit ideal of $\OO_K$. Thus, we have
\begin{equation}\label{equ:B2}
\begin{split}
&\frac{1}{[K':\Q]}\log|\disc(K)|\\
=&\frac{1}{[K':\Q]}\log(\Norm_{K'/\Q}(\D_{K'/\Q}))\\
=&\frac{1}{[K':\Q]}\log(\Norm_{K'/\Q}(\D_{K'/K}\D_{K/E^*_\Phi}\D_{E^*_\Phi/\Q}))\\
=&\frac{1}{[K':\Q]}\log(\Norm_{K'/\Q}(\D_{K'/K}\D_{E^*_\Phi/\Q}))\\
=&\frac{1}{[K':\Q]}\log(\Norm_{K'/\Q}(\D_{K'/K}))+\frac{1}{[K':\Q]}\log(\Norm_{K'/\Q}(\D_{E^*_\Phi/\Q}))\\
=&\frac{1}{[K':\Q]}\log(\Norm_{K'/\Q}(\D_{K'/K}))+\frac{1}{[E^*_\Phi:\Q]}\log|\disc(E^*_\Phi)|\\
\le & 2\cdot \#\mu(E)\log(p)+\frac{1}{[E^*_\Phi:\Q]}\log|\disc(E^*_\Phi)|,
\end{split}
\end{equation}
where the last inequality follows from Equation (\ref{equ:B1}). 

By our assumption on $p$, we have
\[
p\le 70 \cdot \biggl(\log |\disc(E(\mu_{mp_1p_2\cdots p_s}))|\biggr)^2.
\]

Thus, we have
\begin{equation}\label{equ:B3}
\begin{split}
\log(p)&\le \log(70)+2\log\log|\disc(E(\mu_{mp_1p_2\cdots p_s}))|\\
&\le \log(70)+2\log\log \biggl(((4g)^4)^{(4g)^4\cdot 2g}(|\disc(E)|)^{(4g)^4}\biggr) ,
\end{split}
\end{equation}
where the last inequality is by the following \autoref{lem:bound on disc(E(mu_mp1p2...ps))}. 

By the following \autoref{lem:bound on mu(E)}, we have
\begin{equation}\label{equ:B4}
\begin{split}
\mu(E)\le (4g)^2. 
\end{split}
\end{equation}

Plugging Equation (\ref{equ:B3}) and Equation (\ref{equ:B4}) into Equation (\ref{equ:B2}), we get our claim. 
\end{proof}

\begin{lem}\label{lem:bound on mu(E)}
Let $K$ be a number field of degree $[K:\Q]=n$. Let $\mu(K)$ be the group of all roots of unity in $K$. Then $\mu(K)$ is a finite cyclic group of order less than or equal to $(2n)^2$.
\end{lem}
\begin{proof}
Since $[K:\Q]$ is finite, $\mu(K)$ is a finite group. It is easy to see that $\mu(K)$ is a cyclic group. Let $m$ be the order of the group $\mu(K)$, and let $p_1, p_2, \cdots, p_s$ be the distinct prime divisors of $m$. Denote as $\mu_m$ the primitive $m$-th root of unity. Then $K$ contains the $m$-th cyclotomic field $\Q(\mu_m)$. Since $[\Q(\mu_m):\Q]=\#(\Z/m\Z)^{\times}=\frac{m(p_1-1)(p_2-1)\cdots(p_s-1)}{p_1p_2\cdots p_s}$ and $[K:\Q]=n$, this means that
\[
\frac{m(p_1-1)(p_2-1)\cdots(p_s-1)}{p_1p_2\cdots p_s} \le n.
\]

For any prime $p\ne 2$, we have $\sqrt{p} \le p-1$. Thus, we have
\begin{equation*}
\begin{split}
\frac{m(p_1-1)(p_2-1)\cdots(p_s-1)}{p_1p_2\cdots p_s} & \ge \frac{m\sqrt{p_1}\sqrt{p_2}\cdots\sqrt{p_s}}{2p_1p_2\cdots p_s}\\
&=\frac{m}{2\sqrt{p_1}\sqrt{p_2}\cdots\sqrt{p_s}}\\
&\ge \frac{m}{2\sqrt{m}}\\
&=\sqrt{m}/2.
\end{split}    
\end{equation*}

Therefore, we have
\[
m\le (2n)^2.
\]
\end{proof}

\begin{lem}\label{lem:bound on disc(E(mu_mp1p2...ps))}
Let $K$ be a number field of degree $[K:\Q]=n$. Then we have
\[
|\disc(K(\mu_{mp_1p_2\cdots p_s}))| \le ((2n)^4)^{(2n)^4\cdot n}(|\disc(K)|)^{(2n)^4}.
\]
where for any positive integer $k$ $\mu_k$ denotes a primitive $k$-th root of unity, $m$ is the order of the group $\mu(K)$ of all roots of unity in $K$, and $p_1, p_2, \cdots, p_s$ are the distinct prime divisors of $m$.
\end{lem}
\begin{proof}
For any $k\in \Z_{\ge 3}$, the $k$-th cyclotomic field $\Q(\mu_k)$ has degree $[\Q(\mu_k):\Q]=\#(\Z/k\Z)^{\times}$. By \cite{cyclotomicfields}, we have:
\[
\disc(\Q(\mu_k))=(-1)^{\varphi(k)/2}\frac{k^{\varphi(k)}}{\prod_{p|k}p^{\varphi(k)/(p-1)}},
\]
where $\varphi (k)=\#(\Z/k\Z)^{\times}$ is Euler's totient function, and the product in the denominator on the right-hand-side is over primes $p$ dividing $k$.

Thus, we have
\[
|\disc(\Q(\mu_k))| \le k^k.
\]

By \autoref{lem:bound on mu(E)}, we have
\[
m\le (2n)^2.
\]

Thus, the $mp_1p_2\cdots p_s$-th cyclotomic field $\Q(\mu_{mp_1p_2\cdots p_s})$ has degree 
\begin{equation*}
\begin{split}
[\Q(\mu_{mp_1p_2\cdots p_s}):\Q]&=\#(\Z/mp_1p_2\cdots p_s\Z)^{\times}\\
&=m(p_1-1)(p_2-1)\cdots (p_s-1)\\
&\le m^2\\
& \le (2n)^4.
\end{split}    
\end{equation*}

Moreover, $mp_1p_2\cdots p_s\le m^2\le (2n)^4$. Thus, we have
\[
|\disc(\Q(\mu_{mp_1p_2\cdots p_s}))| \le ((2n)^4)^{(2n)^4}.
\]

We know that $K(\mu_{mp_1p_2\cdots p_s})$ is equal to the compositum of $K$ and $\Q(\mu_{mp_1p_2\cdots p_s})$. Thus, by Equation (\ref{equ:X2}), we have
\begin{equation*}
\begin{split}
|\disc(K(\mu_{mp_1p_2\cdots p_s}))| &\le |\disc(K)|^{[\Q(\mu_{mp_1p_2\cdots p_s}):\Q]}|\disc(\Q(\mu_{mp_1p_2\cdots p_s}))|^{[K:\Q]}\\
&\le ((2n)^4)^{(2n)^4\cdot n}(|\disc(K)|)^{(2n)^4}.
\end{split}    
\end{equation*}
\end{proof}

\begin{proof}[Proof of \autoref{prop:field of good reduction can be small}]
The proof is similar to that of \autoref{prop:assume GRH, field of good reduction can be small} (using \autoref{thm:existence of special abelian variety with CM type} and \autoref{thm:field of good reduction}). The difference is that since we do not assume the Generalized Riemann Hypothesis, we use Equation (\ref{equ:bound on p_Phi}) in \autoref{thm:existence of special abelian variety with CM type} to bound the prime $p$ instead of Equation (\ref{equ:bound on p_Phi assuming GRH}). The term $\frac{1}{[E^*_\Phi:\Q]}\log|\disc(E^*_\Phi)|$ of Equation (\ref{equ:bound on disc(K') assuming GRH}) is submerged into the term $C_{15}(g)\log|\disc(E)|$ of Equation (\ref{equ:bound on disc(K')}) by the fact that the reflex field $E^*_\Phi$ is contained in the Galois closure $\widetilde{E}$ of the extension of $E/\Q$ (and so $\frac{1}{[E^*_\Phi:\Q]}\log|\disc(E^*_\Phi)|$ is less than or equal to $\frac{1}{[\widetilde{E}:\Q]}\log(|\disc(\widetilde{E})|)$) and the fact that $\frac{1}{[\widetilde{E}:\Q]}\log(|\disc(\widetilde{E})|)\le \log|\disc(E)|$ by Equation (\ref{equ:X3}). 
\end{proof}

\begin{rem}\label{rem:assume GRH, right hand side of Colmez conjecture}
In comparison with \autoref{prop:assume GRH, field of good reduction can be small}, one might ask how the right-hand-side of the formula in \nameref{thm:averaged Colmez conjecture} behaves under the Generalized Riemann Hypothesis. 
	
The right-hand-side of Equation (\ref{equ:averaged Colmez conjecture}) is equal to 
\begin{equation*}
\begin{split}
&-\frac{1}{2}\frac{L'(0, \chi_{E/F})}{L(0, \chi_{E/F})}-\frac{1}{4}\log(|\disc(E)|/|\disc(F)|) \\
=& \frac{1}{2}\frac{L'(1, \chi_{E/F})}{L(1, \chi_{E/F})}+\frac{1}{4}\log(|\disc(E)|/|\disc(F)|)\\
& +\frac{1}{2}\cdot \frac{g}{2} \Biggl(\frac{\Ga'\bigl(\frac{1}{2}\bigr)}{\Ga\bigl(\frac{1}{2}\bigr)} +\frac{\Ga'(1)}{\Ga(1)}-2\log(\pi) \Biggr),
\end{split}
\end{equation*}
where $g\coloneqq [F:\Q]$. The equality follows from logarithmically differentiating the functional equation of $L(s, \chi_{E/F})$ at $s=0$. 
	
Assume the Generalized Riemann Hypothesis. Then for any $g\in \Z_{\ge 1}$, there exist constants $\CGRHone(g)>0, \CGRHtwo(g)\in \R$ depending only on $g$ such that  
\[
\biggl|\frac{L'(1, \chi_{E/F})}{L(1, \chi_{E/F})} \biggr|\le \CGRHone(g) \log\log|\disc E|+ \CGRHtwo(g)
\]
for any CM-field $E$ with maximal totally real subfield $F$ such that $[F:\Q]=g$. Then it is easy to see that for any $g\in \Z_{\ge 1}$, for any $\epsilon>0$, there exists a constant $c(g, \epsilon)>0$ depending only on $g$ and $\epsilon$ such that 
\begin{equation*}
\begin{split}
&\Biggl|\biggl(-\frac{1}{2}\frac{L'(0, \chi_{E/F})}{L(0, \chi_{E/F})}-\frac{1}{4}\log(|\disc(E)|/|\disc(F)|)\biggr) -\frac{1}{4}\log(|\disc(E)|/|\disc(F)|)\Biggr|\\
<& \epsilon \log|\disc(E)|,
\end{split}
\end{equation*}
for any CM-field $E$ with maximal totally real subfield $F$ such that $[E:\Q]=2g$ and $|\disc(E)|\ge c(g, \epsilon)$.

Since $|\disc(E)|/|\disc(F)|\le |\disc(E)|\le (|\disc(E)|/|\disc(F)|)^2$, this means that assuming the Generalized Riemann Hypothesis, the right-hand-side of Equation (\ref{equ:averaged Colmez conjecture}) is ``approximately some constant times $\log|\disc(E)|$".
\end{rem}

\begin{rem}\label{rem:reflex field can be small}
One might wonder whether there is a \textit{lower} bound for $|\disc(E^*_\Phi)|$ in terms of $|\disc(E)|$ and $[E:\Q]$. The following example shows that the answer is no:

Let $F$ be any totally real number field. Let $-d\in \Z_{\le -2}$ be any fundamental discriminant (so $\disc(\Q(\sqrt{-d}))=d$) such that $-d$ is prime to $\disc(F)$. (For any totally real number field $F$, there are infinitely many such $-d$.) Let $E$ be the compositum of the fields $F$ and $\Q(\sqrt{-d})$.

Then $E$ is a CM-field with maximal totally real subfield $F$. Let $\Phi$ be the CM-type defined as follows:

For any $\varphi_0\in \Hom_\Q(F, \R)$, the element $\phi\colon E\to \C$ in $\Phi$ lying above $\varphi_0$ always sends $\sqrt{-d}$ to $\sqrt{-d}$.							

Then it is easy to see that $E^*_\Phi=\Q(\sqrt{-d})$. Thus, $\disc(E^*_\Phi)=\disc(\Q(\sqrt{-d}))=d$. 

Since $\disc(\Q(\sqrt{-d}))=d$ is coprime to $\disc(F)$, by Theorem 4.26 of \cite{Algebraicnumbers}, for example, we have
\begin{equation*}
\begin{split}
|\disc(E)|=d^{[F:\Q]} |\disc(F)|^2. 
\end{split}
\end{equation*}

Therefore, for any fixed $g\in \Z_{\ge 2}$, the quotient 
\[
\frac{\log|\disc(E^*_\Phi)|}{\log|\disc(E)|},
\]
where $E$ is a CM-field of degree $[E:\Q]=2g$ and $\Phi$ is a CM-type of $E$, can be arbitrarily small. 
\end{rem}

Combining \autoref{rem:reflex field can be small} with \autoref{prop:assume GRH, field of good reduction can be small}, we have shown the following:

\begin{prop}\label{prop:field of good reduction is much smaller than disc(E)}
Assume the Generalized Riemann Hypothesis. For any $g\in \Z$ such that \underline{$g\ge 2$}, for any $\epsilon>0$, there exists a CM-field $E$ with $[E:\Q]=2g$, a CM-type $\Phi$ of $E$, a number field $K'$ and a CM abelian variety $(A, i\colon \OO_E\into \End_{K'}(A))$ over $K'$ of CM-type $\Phi$ such that the abelian variety $A$ over $K'$ has everywhere good reduction and 
\begin{equation*}
\begin{split}
\frac{1}{[K':\Q]}\log|\disc(K')|\le \epsilon \log|\disc(E)|.
\end{split}
\end{equation*}
\end{prop}

\begin{rem}
In view of \autoref{rem:assume GRH, right hand side of Colmez conjecture} and \autoref{prop:field of good reduction is much smaller than disc(E)}, we cannot remove the ``average" condition in \autoref{thm:2} and \autoref{thm:4}---Using only the \nameref{thm:averaged Colmez conjecture}, we can only prove \textit{averaged} analogues of \autoref{thm:1} and \autoref{thm:3}. 
\end{rem}

\begin{rem}
In Theorem 6(i) of \cite{ColmezFaltingsheight}, Colmez has proved that there exist effectively computable absolute constants $\CColthree>0$, $\CColfour\in \R$ such that for any CM-field $E$ of degree $[E:\Q]=2g$ and any CM-type $\Phi$ of $E$ such that the following hold:
\begin{enumerate}
\item  $(E, \Phi)$ satisfies the Colmez conjecture,
\item  For any irreducible Artin character $\chi$ such that $\Colm(\chi)\ne 0$, the Artin conjecture for $\chi$ holds (i.e. the Artin $L$-function $L(s, \chi, \Q)$ is holomorphic everywhere except possibly for a simple pole at $s=1$),
\end{enumerate}
we have
\begin{equation*}
\begin{split}
\hFaltE \ge \CColthree \cdot \mu_{(E, \Phi)}+g \CColfour .
\end{split}
\end{equation*}

Let $E$ be a CM-field of degree $[E:\Q]=2g$ and let $\Phi$ be a CM-type of $E$. It is easy to see that for the function $\ColA$ from $\Gal(\NE/\Q)$ to $\C$, for any $\si\in \Gal(\NE/\Q)$, $\ColA(\si)=g$ if and only if $\si=1$. Therefore, some calculations using the definition of the Artin conductor of Artin characters show that for any $g\in \Z_{\ge 1}$, there exist effectively computable constants $C_{\mu,1}(g)>0$, $C_{\mu,2}(g)\in \R$ such that
\[
\mu_{(E, \Phi)} \ge C_{\mu,1}(g) \frac{1}{[E^*_\Phi:\Q]}\log|\disc(E^*_\Phi)|+C_{\mu,2}(g)
\]
for any CM-field $E$ of degree $[E:\Q]=2g$ and any CM-type $\Phi$ of $E$. We can compare this to \autoref{prop:assume GRH, field of good reduction can be small}. 
\end{rem}
\bibliographystyle{alpha}
\bibliography{sample}

\begin{thebibliography}{AGHMP18}

\bibitem[AGHMP18]{AGHM}
Fabrizio Andreatta, Eyal~Z. Goren, Benjamin Howard, and Keerthi Madapusi~Pera.
\newblock Faltings heights of abelian varieties with complex multiplication.
\newblock {\em Ann. of Math. (2)}, 187(2):391--531, 2018.

\bibitem[Art86]{ArithGeoArtin}
M.~Artin.
\newblock N\'{e}ron models.
\newblock In {\em Arithmetic geometry ({S}torrs, {C}onn., 1984)}, pages
  213--230. Springer, New York, 1986.

\bibitem[BK94]{Conductor}
Armand Brumer and Kenneth Kramer.
\newblock The conductor of an abelian variety.
\newblock {\em Compositio Math.}, 92(2):227--248, 1994.

\bibitem[BLR90]{Neronmodel}
Siegfried Bosch, Werner L\"{u}tkebohmert, and Michel Raynaud.
\newblock {\em N\'{e}ron models}, volume~21 of {\em Ergebnisse der Mathematik
  und ihrer Grenzgebiete (3) [Results in Mathematics and Related Areas (3)]}.
\newblock Springer-Verlag, Berlin, 1990.

\bibitem[Bos96]{Bost}
Jean-Beno\^{\i}t Bost.
\newblock Arakelov geometry of abelian varieties.
\newblock In {\em Conference on Arithmetical Geometry}, Berlin, March 21–26
  1996. Max-Planck Institut für Mathematik Bonn.
\newblock Preprint 96–51.

\bibitem[CCO14]{CCO}
Ching-Li Chai, Brian Conrad, and Frans Oort.
\newblock {\em Complex multiplication and lifting problems}, volume 195 of {\em
  Mathematical Surveys and Monographs}.
\newblock American Mathematical Society, Providence, RI, 2014.

\bibitem[Cha86]{ArithGeoChai}
Ching-Li Chai.
\newblock Siegel moduli schemes and their compactifications over {${\bf C}$}.
\newblock In {\em Arithmetic geometry ({S}torrs, {C}onn., 1984)}, pages
  231--251. Springer, New York, 1986.

\bibitem[Col93]{Colmez}
Pierre Colmez.
\newblock P\'{e}riodes des vari\'{e}t\'{e}s ab\'{e}liennes \`a multiplication
  complexe.
\newblock {\em Ann. of Math. (2)}, 138(3):625--683, 1993.

\bibitem[Col98]{ColmezFaltingsheight}
Pierre Colmez.
\newblock Sur la hauteur de {F}altings des vari\'{e}t\'{e}s ab\'{e}liennes \`a
  multiplication complexe.
\newblock {\em Compositio Math.}, 111(3):359--368, 1998.

\bibitem[Dav80]{Davenport}
Harold Davenport.
\newblock {\em Multiplicative number theory}, volume~74 of {\em Graduate Texts
  in Mathematics}.
\newblock Springer-Verlag, New York-Berlin, second edition, 1980.
\newblock Revised by Hugh L. Montgomery.

\bibitem[Fal83]{Faltings}
G.~Faltings.
\newblock Endlichkeitss\"{a}tze f\"{u}r abelsche {V}ariet\"{a}ten \"{u}ber
  {Z}ahlk\"{o}rpern.
\newblock {\em Invent. Math.}, 73(3):349--366, 1983.

\bibitem[GS00]{GranvilleStark}
Andrew Granville and H.~M. Stark.
\newblock {$abc$} implies no ``{S}iegel zeros'' for {$L$}-functions of
  characters with negative discriminant.
\newblock {\em Invent. Math.}, 139(3):509--523, 2000.

\bibitem[Hin07]{Hindry}
Marc Hindry.
\newblock Why is it difficult to compute the {M}ordell-{W}eil group?
\newblock In {\em Diophantine geometry}, volume~4 of {\em CRM Series}, pages
  197--219. Ed. Norm., Pisa, 2007.

\bibitem[How18]{Howard}
Benjamin Howard.
\newblock On the averaged colmez conjecture, 2018.

\bibitem[Lan59]{LangAbVar}
Serge Lang.
\newblock {\em Abelian varieties}.
\newblock Interscience Tracts in Pure and Applied Mathematics, No. 7.
  Interscience Publishers, Inc., New York; Interscience Publishers Ltd.,
  London, 1959.

\bibitem[Lev96]{Levin}
B.~Ya. Levin.
\newblock {\em Lectures on entire functions}, volume 150 of {\em Translations
  of Mathematical Monographs}.
\newblock American Mathematical Society, Providence, RI, 1996.
\newblock In collaboration with and with a preface by Yu. Lyubarskii, M. Sodin
  and V. Tkachenko, Translated from the Russian manuscript by Tkachenko.

\bibitem[Mil72]{MilneOnthearithmetic}
J.~S. Milne.
\newblock On the arithmetic of abelian varieties.
\newblock {\em Invent. Math.}, 17:177--190, 1972.

\bibitem[Mil05]{MilneShimVar}
J.~S. Milne.
\newblock Introduction to {S}himura varieties.
\newblock In {\em Harmonic analysis, the trace formula, and {S}himura
  varieties}, volume~4 of {\em Clay Math. Proc.}, pages 265--378. Amer. Math.
  Soc., Providence, RI, 2005.

\bibitem[Mil08]{MilneAbVar}
James~S. Milne.
\newblock Abelian varieties (v2.00), 2008.
\newblock Available at www.jmilne.org/math/.

\bibitem[Mil17]{MilneAlgebraicgroups}
J.~S. Milne.
\newblock {\em Algebraic groups}, volume 170 of {\em Cambridge Studies in
  Advanced Mathematics}.
\newblock Cambridge University Press, Cambridge, 2017.
\newblock The theory of group schemes of finite type over a field.

\bibitem[Mil20]{MilneCLassFieldTheory}
J.S. Milne.
\newblock Class field theory (v4.03), 2020.
\newblock Available at www.jmilne.org/math/.

\bibitem[MM97]{Murty}
M.~Ram Murty and V.~Kumar Murty.
\newblock {\em Non-vanishing of {$L$}-functions and applications}.
\newblock Modern Birkh\"{a}user Classics. Birkh\"{a}user/Springer Basel AG,
  Basel, 1997.
\newblock [2011 reprint of the 1997 original] [MR1482805].

\bibitem[MP05]{cyclotomicfields}
Yuri~Ivanovic Manin and Alexei~A. Panchishkin.
\newblock {\em Introduction to modern number theory}, volume~49 of {\em
  Encyclopaedia of Mathematical Sciences}.
\newblock Springer-Verlag, Berlin, second edition, 2005.
\newblock Fundamental problems, ideas and theories, Translated from the
  Russian.

\bibitem[N\'64]{Neron}
Andr\'{e} N\'{e}ron.
\newblock Mod\`eles minimaux des vari\'{e}t\'{e}s ab\'{e}liennes sur les corps
  locaux et globaux.
\newblock {\em Inst. Hautes \'{E}tudes Sci. Publ. Math.}, (21):128, 1964.

\bibitem[Nar90]{Algebraicnumbers}
Władysław Narkiewicz.
\newblock {\em Elementary and analytic theory of algebraic numbers}.
\newblock Springer-Verlag, Berlin; PWN---Polish Scientific Publishers, Warsaw,
  second edition, 1990.

\bibitem[Obu13]{Obus}
Andrew Obus.
\newblock On {C}olmez's product formula for periods of {CM}-abelian varieties.
\newblock {\em Math. Ann.}, 356(2):401--418, 2013.

\bibitem[Ogg67]{Ogg}
A.~P. Ogg.
\newblock Elliptic curves and wild ramification.
\newblock {\em Amer. J. Math.}, 89:1--21, 1967.

\bibitem[Ser79]{Localfields}
Jean-Pierre Serre.
\newblock {\em Local fields}, volume~67 of {\em Graduate Texts in Mathematics}.
\newblock Springer-Verlag, New York-Berlin, 1979.
\newblock Translated from the French by Marvin Jay Greenberg.

\bibitem[Shi98]{Sh}
Goro Shimura.
\newblock {\em Abelian varieties with complex multiplication and modular
  functions}, volume~46 of {\em Princeton Mathematical Series}.
\newblock Princeton University Press, Princeton, NJ, 1998.

\bibitem[ST68]{goodreduction}
Jean-Pierre Serre and John Tate.
\newblock Good reduction of abelian varieties.
\newblock {\em Ann. of Math. (2)}, 88:492--517, 1968.

\bibitem[Sta74]{Stark}
H.~M. Stark.
\newblock Some effective cases of the {B}rauer-{S}iegel theorem.
\newblock {\em Invent. Math.}, 23:135--152, 1974.

\bibitem[Tat84]{Tat84}
John Tate.
\newblock {\em Les conjectures de {S}tark sur les fonctions {$L$} d'{A}rtin en
  {$s=0$}}, volume~47 of {\em Progress in Mathematics}.
\newblock Birkh\"{a}user Boston, Inc., Boston, MA, 1984.
\newblock Lecture notes edited by Dominique Bernardi and Norbert Schappacher.

\bibitem[Wei76]{WeilElliptic}
Andr\'{e} Weil.
\newblock {\em Elliptic functions according to {E}isenstein and {K}ronecker}.
\newblock Ergebnisse der Mathematik und ihrer Grenzgebiete, Band 88.
  Springer-Verlag, Berlin-New York, 1976.

\bibitem[Yua19]{Yuan}
Xinyi Yuan.
\newblock On {F}altings heights of abelian varieties with complex
  multiplication.
\newblock In {\em Proceedings of the {S}eventh {I}nternational {C}ongress of
  {C}hinese {M}athematicians, {V}ol. {I}}, volume~43 of {\em Adv. Lect. Math.
  (ALM)}, pages 521--536. Int. Press, Somerville, MA, 2019.

\bibitem[YZ18]{YuanZhang}
Xinyi Yuan and Shou-Wu Zhang.
\newblock On the averaged {C}olmez conjecture.
\newblock {\em Ann. of Math. (2)}, 187(2):533--638, 2018.

\end{thebibliography}
\nocite{*}
\end{document}